\documentclass[12pt]{iopart}

\usepackage{iopams}
\usepackage{nicefrac}

\usepackage{slashbox}

\usepackage{graphicx}
\usepackage{subfigure}

\def\N			{\mathbb N}
\def\Z			{\mathbb Z}

\def\R			{\mathbb R}

\def\Radon		{\mathcal R}
\def\Back		{\mathcal B}
\def\Fourier	{\mathcal F}
\def\Int		{\mathcal I}

\def\L			{\mathrm L}

\def\B			{\mathrm B}

\def\Cont		{\mathcal C}

\def\d			{\mathrm d}
\def\e			{\mathrm e}
\def\i			{\mathrm i}

\def\bfx		{\mathbf x}

\def\loc		{\mathrm{loc}}

\def\supp		{\mathrm{supp}}
\def\diam		{\mathrm{diam}}

\def\sinc		{\mathrm{sinc}}

\def\FBP		{\mathrm{FBP}}

\newcounter{mycounter}

\newtheorem{lemma}[mycounter]{Lemma}
\newtheorem{theorem}[mycounter]{Theorem}
\newtheorem{corollary}[mycounter]{Corollary}

\newtheorem{example}[mycounter]{Example}

\newenvironment{proof}{\paragraph{\textit{Proof}}}{\hfill$\square$\bigbreak}

\begin{document}

\title[Error analysis for FBP reconstructions in Besov spaces]{Error analysis for filtered back projection reconstructions in Besov spaces}

\author{M Beckmann$^1$, P Maass$^2$ and J Nickel$^2$}
\address{$^1$ Department of Mathematics, University of Hamburg, Hamburg, Germany}
\address{$^2$ Center for Industrial Mathematics, University of Bremen, Bremen, Germany}
\eads{\mailto{matthias.beckmann@uni-hamburg.de}, \mailto{pmaass@math.uni-bremen.de}, \mailto{junickel@uni-bremen.de}}

\begin{abstract}
Filtered back projection (FBP) methods are the most widely used reconstruction algorithms in computerized tomography (CT).
The ill-posedness of  this inverse problem allows only an approximate reconstruction for given noisy data.
Studying the resulting reconstruction error has been a most active field of research in the 1990s and has recently 
been revived in terms of optimal filter design and estimating the FBP approximation errors in general Sobolev spaces.

However, the choice of Sobolev spaces is suboptimal for characterizing typical CT reconstructions.
A widely used model are sums of characteristic functions, which are better modelled in terms of Besov spaces $\B^{\alpha,p}_q(\R^2)$.
In particular $\B^{\alpha,1}_1(\R^2)$ with $\alpha \approx 1$ is a preferred model in image analysis for describing natural images.

In case of noisy Radon data the total FBP reconstruction error 
$$\|f-f_L^\delta\| \le \|f-f_L\|+ \|f_L - f_L^\delta\|$$
splits into an approximation error and a data error, where $L$ serves as regularization parameter.
In this paper, we study the approximation error of FBP reconstructions 
for target functions $f \in \L^1(\R^2) \cap \B^{\alpha,p}_q(\R^2)$ with positive $\alpha \not\in \N$ and $1 \leq p,q \leq \infty$.  
We prove that the $\L^p$-norm of the inherent FBP approximation error $f-f_L$ can be bounded above by
\begin{equation*}
\|f - f_L\|_{\L^p(\R^2)} \leq c_{\alpha,q,W} \, L^{-\alpha} \, |f|_{\B^{\alpha,p}_q(\R^2)}
\end{equation*}
under suitable assumptions on the utilized low-pass filter's window function $W$. This then extends by classical methods to estimates for the total reconstruction error.
\end{abstract}


\section{Introduction}

We consider the classical inverse problem of reconstructing a function $f: \Omega \rightarrow \R$, $\Omega \subset \R^2$, from its line integrals, which is  the mathematical model underlying X-ray computerized tomography (CT).
The line integrals of $f$ are defined by the {\em Radon transform} $\Radon f \equiv \Radon f(t,\theta)$, given by
\begin{equation*}
\Radon f(t,\theta) = \int_{\ell_{t,\theta}} f(x,y) \: \mathrm{d}(x,y)
\quad \mbox{ for } (t,\theta) \in \R \times [0,\pi),
\end{equation*}
where the set
\begin{equation*}
\ell_{t,\theta} = \bigl\{(x,y) \bigm| x\cos(\theta)+y\sin(\theta) = t \bigr\} \subset \R^2
\end{equation*}
denotes the unique straight line that is orthogonal to the unit vector $n_\theta = (\cos(\theta),\sin(\theta))$ and has (signed) distance~$t$ to the origin, i.e., $\ell_{t,\theta}$ passes through $(t \cos(\theta), t\sin(\theta)) \in \R^2$.

For $f \in \L^1(\R^2) \cap \Cont(\R^2)$ with $\Fourier f \in \L^1(\R^2)$ an analytical inversion formula is given by
\begin{equation*}
f = \frac{1}{2} \, \Back I  \left(  \Radon f \right).
\end{equation*}
Here $\Back$ denotes the back projection operator, which is the $\L^2$-adjoint of $\Radon$ given by
\begin{equation*}
\Back h(x,y) = \frac{1}{\pi} \int_0^{\pi} h(x\cos(\theta)+y\sin(\theta),\theta) \: \mathrm{d}\theta
\quad \mbox{ for } (x,y) \in \R^2,
\end{equation*}
and $I$ denotes the Riesz potential defined via the one-dimensional Fourier transform acting on the $t$-variable, i.e.,
\begin{equation*}
 \Fourier g(S,\theta) = \int_{\R} g(t,\theta) \, \mathrm{e}^{-\mathrm{i} t S} \: \mathrm{d}t
\quad \mbox{ for } (S,\theta) \in \R \times [0,\pi)
\end{equation*}
and
\begin{equation*}
\Fourier \left( I  g \right) (S,\theta) = |S| \Fourier g(S,\theta).
\end{equation*}
The assumption of a continuous $f$ is necessary in order to ensure that the inversion formula holds pointwise, cf.~\cite{Natterer2001,Beckmann2019a}. 
However, the continuity assumption is not needed for error estimates concerning  regularized reconstruction algorithms as discussed in the sequel of this paper.

The analytical inversion is unstable with respect to highly oscillating variations of $g = \Radon f$, which motivates the introduction of a regularized inversion formula, the so-called method of filtered back projection (FBP).
FBP is based on introducing a window function
\begin{eqnarray*}
W: & \R & \to \R \\
& S & \mapsto W(S)
\end{eqnarray*}
with $\|W\|_\infty < \infty$, which has either bounded support or decays fast enough at infinity.
The window function $W$ is scaled by a parameter $L$ and we introduce the low-pass filter
\begin{equation*}
A_L(S) = |S| \, W(\nicefrac{S}{L})
\quad \mbox{ for } S \in \R
\end{equation*}
to replace the Fourier multiplier $|S|$ in the Riesz potential leading to the {\em approximate} FBP formula
\begin{equation*}
f_L(x,y) = \frac{1}{2} \, \Back\big(\Fourier^{-1}[A_L(S) \Fourier(\Radon f)(S,\theta)]\big)(x,y)
\quad \mbox{ for } (x,y) \in \R^2 \ .
\end{equation*}
Here, $L$ serves as regularization parameter and will be adapted depending on the noise level in the data.
Window functions of typical low-pass filters are displayed in Table~\ref{tab:lpf}.

\begin{table}[b]
\centering
\begin{tabular}{l|c|c}
Name & $W(S)$ for $|S|\leq 1$ & Parameter\\
\hline
Ram-Lak & $1$ & - \\
Shepp-Logan & $\sinc(\nicefrac{\pi S}{2})$ & - \\
Cosine & $\cos(\nicefrac{\pi S}{2})$ & - \\
Hamming & $\beta + (1-\beta) \cos(\pi S)$ & $\beta \in [\nicefrac{1}{2},1]$ \\
%Gaussian & $\exp(-(\nicefrac{\pi S}{\beta})^2)$ & $\beta > 1$ \\
\end{tabular}
\caption{Window functions of commonly used low-pass filters, where $W(S) = 0$ for all $|S| > 1$ in all cases.}
\label{tab:lpf}
\end{table}

\bigbreak

In case of noisy data $g^\delta$ with $\|g^\delta - \Radon f\| \le \delta$ we compute the reconstruction 
\begin{equation*}
f_L^\delta(x,y) = \frac{1}{2} \, \Back\big(\Fourier^{-1}[A_L(S) \Fourier(g^\delta)(S,\theta)]\big)(x,y)
\end{equation*}
and the total reconstruction error 
\begin{equation*}
\|f-f_L^\delta\| \le \|f - f_L\|+ \|f_L - f_L^\delta\|
\end{equation*}
splits into an approximation error
\begin{equation*}
e_L = f-f_L,
\end{equation*}
whose analysis is the main target of the present paper, and a data error $f_L - f_L^\delta$.

Such FBP methods are the most widely used reconstruction algorithms in computerized tomography (CT).
Studying the resulting reconstruction error has been a most active field of research in the 1990s and has recently been revived in terms of optimal filter design and estimating the approximation errors in general Sobolev spaces \cite{Beckmann2019,Beckmann2019a,Beckmann2020}.
We will review the state of research in more detail in the next section.

For motivation of the present paper, we note that the choice of Sobolev spaces is suboptimal for characterizing typical CT reconstructions.
A widely used model for CT reconstructions are sums of characteristic functions, which are better modelled in terms of Besov spaces $\B^{\alpha,p}_q(\R^2)$.
In particular, $\B^{\alpha,1}_1(\R^2)$ with $\alpha$ close to $1$ is a preferred model in image analysis for describing natural images \cite{Lucier2019,Pinzon2001}.

Therefore, in this paper we focus on extending the analysis of the approximation error of FBP reconstructions 
to target functions $f \in \L^1(\R^2) \cap \B^{\alpha,p}_q(\R^2)$ with positive $\alpha \not\in \N$ and $1 \leq p,q \leq \infty$.  
We prove that the $\L^p$-norm of the inherent FBP approximation error $f-f_L$ can be bounded above by
\begin{equation*}
\|f - f_L\|_{\L^p(\R^2)} \leq c_{\alpha,q,W} \, L^{-\alpha} \, |f|_{\B^{\alpha,p}_q(\R^2)}
\end{equation*}
under suitable assumptions on the utilized low-pass filter's window function $W$. This then extends by classical methods to estimates for the total reconstruction error.

The transition from error estimates in Sobolev spaces to Besov spaces requires substantially different techniques. 
The Sobolev space estimates in \cite{Beckmann2019,Beckmann2019a,Beckmann2020} are implicitly based on Plancherel's formula, which states that the Fourier transform is an isometry between $\L^2$-spaces, which is not  the case for $\L^p$-spaces with $p \ne 2$.
Hence, the definition of Besov spaces in terms of moduli of smoothness requires different analytical tools for estimating the reconstruction error with respect to the target function's Besov norm.
  
The paper is organized as follows. In Section 2 we review the state of the art concerning approximation errors of FBP reconstruction for functions on unbounded domains.
We then introduce the definition of Besov spaces used in this paper along with some technical Lemmata, which will be needed later for estimating $\|f - f_L\|_{\L^p(\R^2)}$.

Section 4 then contains the main results of the paper. The proofs are split into the cases $0 < \alpha < 1, 1 \le p,q \le \infty$ and
$1 < \alpha , \alpha \notin \N, 1 \le p,q \le \infty$.
We also include a straight forward result on how these results on $e_L = f - f_L$ extend to an estimate of the total approximation error for noisy data.

Section 5 contains some numerical experiments confirming the theoretical findings of the previous section.

\section{State of the art}

Although the FBP method has been one of the standard reconstruction algorithms in CT for decades, its error analysis and convergence behaviour are not completely settled so far.
We shortly summarize the available literature on estimating total reconstruction and approximation errors for FBP reconstructions.
For a general reference we refer to the standard textbooks \cite{Natterer2001,Natterer2001a} and to the introductory chapters of \cite{Beckmann2019a}, which contains an in-depth description and comparison of the results by Madych, which are most relevant for our approach.
Indeed, the description of the state of the art in \cite{Beckmann2019a} serves as the main reference for the following summary.

FBP algorithms and their approximation properties were explicitely or at least implicitely addressed already in the very first papers and textbooks \cite{Herman2009,Natterer2001} on the mathematics of computerized tomography.
Arguably the first paper addressing an analysis of $e_L = f - f_L$ in a classical function space setting is~\cite{Popov1990}.
There, Popov showed pointwise convergence however with a restriction to a small class of piecewise smooth functions.
Pointwise convergence and $\L^\infty$-error estimates for $e_L$ are also discussed by Munshi et al.\ in~\cite{Munshi1991,Munshi1992}. 
Their results are supported by numerical experiments in~\cite{Munshi1993}.

The approach of Rieder and Schuster~\cite{Rieder2003a} leads to $\L^2$-convergence with suboptimal rates for compactly supported Sobolev functions.
In contrast to this, in~\cite{Rieder2003,Rieder2007} Rieder et al.\ prove optimal $\L^2$-convergence rates for sufficiently smooth Sobolev functions.
However, the authors verify their assumptions only for a restricted class of filters based on B-splines.
More recently, Qu~\cite{Qu2016} showed convergence without rates in the $\L^2$-norm for compactly supported $\L^\infty$-functions and in points of continuity under additional assumptions.
Note that~\cite{Qu2016} deals with the continuous problem, while~\cite{Rieder2003,Rieder2003a,Rieder2007} discuss discrete settings.

More relevant for our present paper is the approach described by Madych, who proves
error bounds on the $\L^p$-norm of $e_L$ in terms of $\L^p$-moduli of continuity of the target function $f$, see~\cite{Madych1990}.
Madych chooses a convolution kernel $K: \R^2 \to \R$ as an approximation of the identity and computes the convolution product $f*K_L$ to approximate the target function $f$, where, for $L>0$, the scaled kernel $K_L$ is given by
\begin{equation*}
K_L(x,y) = L^2 \, K(Lx,Ly)
\quad \mbox{ for } (x,y) \in \R^2.
\end{equation*}
If $K$ is chosen to be a uniform sum of ridge functions, the convolution $f*K_L$ can be expressed in terms of the Radon data $\Radon f$ as in the approximate FBP formula~(\ref{eq:FBP_approximate}), see~\cite[Proposition 1]{Madych1990}.
The assumptions on $K$ for proving these results are rather restrictive, in particular they require continuous filter functions, which e.g. excludes the case of a ramp filter.
Using an essentially different approach, \cite{Beckmann2019,Beckmann2019a,Beckmann2020} then proved Sobolev space estimates in a more general setting with substantially weaker assumptions on the filter, including all classical choices.

To some extend, the approach of Madych is a special case of the mollifier approaches used in \cite{Louis1996,Louis1990,Rieder2000}.
However, neither Sobolev nor Besov space error estimates for the continuous case are derived in these papers.

\section{Besov spaces and technical Lemmata}
  
The focus of the paper is on analysing the $\L^p$-norm of the inherent FBP reconstruction error $e_L = f - f_L$ for target functions $f \in \L^1(\R^2) \cap \B^{\alpha,p}_q(\R^2)$ with positive $\alpha \not\in \N$ and $1 \leq p,q \leq \infty$, where
\begin{equation*}
\B^{\alpha,p}_q(\R^2) = \bigl\{f \in \L^p(\R^2) \bigm| |f|_{\B^{\alpha,p}_q(\R^2)} < \infty\bigl\}
\end{equation*}
with
\begin{equation*}
|f|_{\B^{\alpha,p}_q(\R^2)} = \cases{
\left(\int_0^\infty \big(t^{-\alpha} \, \omega_p(f,t)\big)^q \: \frac{\d t}{t}\right)^{\nicefrac{1}{q}} & for $1 \leq q < \infty$, \\
\sup_{t>0} t^{-\alpha} \, \omega_p(f,t) & for $q = \infty$}
\end{equation*}
for $0 < \alpha < 1$, where
\begin{equation*}
\omega_p(f,\delta) = \sup_{\|(X,Y)\|_{\R^2} \leq \delta} \|f(\cdot-X,\cdot-Y) - f\|_{\L^p(\R^2)}
\quad \mbox{ for } \delta > 0,
\end{equation*}
and
\begin{equation*}
|f|_{\B^{\alpha,p}_q(\R^2)} = \sum_{|\boldsymbol{j}| = n} \Bigl(\begin{array}{c} n \\ \boldsymbol{j} \end{array}\Bigr) \, |f^{(\boldsymbol{j})}|_{\B^{\theta,p}_q(\R^2)} = \sum_{j_1+j_2 = n} \frac{n!}{j_1! j_2!} \, |f^{(j_1,j_2)}|_{\B^{\theta,p}_q(\R^2)}
\end{equation*}
for $\alpha = n + \theta$ with $n \in \N$ and $0 < \theta < 1$.
As a general reference for properties of Besov spaces we refer to \cite{Leoni2017}.

We start by proving some technical Lemmata, which will be needed in the subsequent sections for estimating $e_L$. The critical part in these estimates is to control the $\L^p$-, resp.\ $\L^\infty$-norm of the modulus of smoothness in the definition of the Besov semi-norm.

We start with an $\L^p$-estimate.
\begin{lemma}
\label{lem:embedding_p}
Let $1 \leq q < p < \infty$ and $\alpha > 0$.
Further, let $g: [0,\infty) \to [0,\infty)$ be an increasing function.
Then, for any $c > 1$,
\begin{equation*}
\left(\int_0^\infty \big(t^{-\alpha} \, g(t)\big)^p \: \frac{\d t}{t} \right)^{\nicefrac{1}{p}} \leq c^{2\alpha} \, \log(c)^{\nicefrac{1}{p}-\nicefrac{1}{q}} \left(\int_0^\infty \big(t^{-\alpha} \, g(t)\big)^q \: \frac{\d t}{t} \right)^{\nicefrac{1}{q}}.
\end{equation*}
\end{lemma}

The proof of this Lemma is mostly technical and has been moved to Appendix A.
We now use the previous Lemma for proving an $\L^\infty$-estimate, which is equivalent to an embedding into the $\L^p$-setting.

\begin{lemma}
\label{lem:embedding_infinity}
Let $1 \leq q < \infty$ and $\alpha > 0$.
Further, let $g: [0,\infty) \to [0,\infty)$ be an increasing function.
Then, for any $c > 1$,
\begin{equation*}
\sup_{t > 0} t^{-\alpha} \, g(t) \leq c^{2\alpha} \, \log(c)^{-\nicefrac{1}{q}} \left(\int_0^\infty \big(t^{-\alpha} \, g(t)\big)^q \: \frac{\d t}{t} \right)^{\nicefrac{1}{q}}.
\end{equation*}
\end{lemma}
The proof of this lemma can be found in Appendix B.
The classical estimates are only concerned with fixed $\alpha$ and qualitative estimates  of the  constant   involved in the Besov-norm estimates. However, the asymptotic behaviour for $\alpha \searrow 0$ is needed for a refined analysis in the next section.

\begin{lemma}
\label{lem:limit}
Let $1 \leq q < \infty$ and let $g: [0,\infty) \to [0,\infty)$ be increasing and bounded from above.
Further, assume that there exists some $\sigma \in (0,1)$ such that
\begin{equation*}
\left(\int_0^\infty \big(t^{-\sigma} \, g(t)\big)^q \: \frac{\d t}{t} \right)^{\nicefrac{1}{q}} < \infty.
\end{equation*}
Then,
\begin{equation*}
\lim_{\alpha \searrow 0} \left(\alpha q \int_0^\infty \big(t^{-\alpha} \, g(t)\big)^q \: \frac{\d t}{t} \right)^{\nicefrac{1}{q}} = \lim_{t \to \infty} g(t).
\end{equation*}
\end{lemma}
Again, the proof of this lemma  has been moved to the appendix, see Appendix C.

\section{Approximation error in Besov spaces}

We now turn to estimating the approximation error $e_L = f-f_L$ of FBP reconstructions  in $\L^p$-norms under the assumption that $f \in \L^1(\R^2) \cap \B^{\alpha,p}_q(\R^2)$. We will  discuss the cases $0 < \alpha <1$ and $\alpha >1$ separately.
These estimates are then used for deriving a bound on the total FBP reconstruction error $f -f_L^\delta$ for noisy data $g^\delta$.
  
As already stated, we consider the {\em approximate} filtered back projection (FBP) formula
\begin{equation*}
f_L(x,y) = \frac{1}{2} \, \Back\big(\Fourier^{-1}[A_L(S) \Fourier(\Radon f)(S,\theta)]\big)(x,y)
\quad \mbox{ for } (x,y) \in \R^2
\end{equation*}
with a given low-pass filter
\begin{equation*}
A_L(S) = |S| \, W(\nicefrac{S}{L})
\quad \mbox{ for } S \in \R
\end{equation*}
of finite bandwidth $L > 0$ and with even window function $W \in \L^\infty(\R)$ satisfying
\begin{equation*}
|\cdot| \, W(\cdot) \in \L^1(\R) \cap \L^2(\R).
\end{equation*}

Recall that for target functions $f \in \L^1(\R^2)$ the approximate FBP reconstruction $f_L \in \L^1_\loc(\R^2)$ is defined almost everywhere on $\R^2$ and can be rewritten as
\begin{equation}
\label{eq:FBP_approximate}
f_L = \frac{1}{2} \, \Back\big(\Fourier^{-1} A_L * \Radon f\big) = f * K_L,
\end{equation}
where the convolution kernel $K_L \in \L^2(\R^2)$ is given by
\begin{equation*}
K_L(x,y) = \frac{1}{2} \, \Back \big(\Fourier^{-1} A_L\big)(x,y)
\quad \mbox{ for } (x,y) \in \R^2.
\end{equation*}

\subsection{Error Estimate for $0 < \alpha < 1$}

Several papers, see e.g. \cite{Lucier2019,Pinzon2001}, have argued, that natural images including cross sections of the human body can be modelled by Besov spaces with $\alpha < 1$. This also includes the case of functions which are superpositions of characteristic functions of smooth domains, which serves as a standard model for simulation in tomography.

Hence, we start with analyzing the case $1 \leq p,q \leq \infty$ and $0 < \alpha < 1$.

\begin{theorem}
\label{theo:Besov_small}
Let $f \in \L^1(\R^2) \cap \B^{\alpha,p}_q(\R^2)$ for $1 \leq p,q \leq \infty$ and $0 < \alpha < 1$.
Furthermore, let $W \in \L^\infty(\R)$ be even with $W(0)=1$ such that the corresponding filter $A \equiv A_1$ satisfies $A \in \L^1(\R) \cap \L^2(\R)$ and the convolution kernel $K \equiv K_1$ satisfies $K \in \L^1(\R^2)$ as well as
\begin{equation*}
\int_{\R^2} \|(x,y)\|_{\R^2}^\alpha \, |K(x,y)| \: \d (x,y) < \infty.
\end{equation*}
Then, the $\L^p$-norm of the inherent FBP reconstruction error $e_L = f - f_L$ is bounded above by
\begin{equation*}
\|e_L\|_{\L^p(\R^2)} \leq c_{\alpha,q} \left(\int_{\R^2} \|(x,y)\|_{\R^2}^\alpha \, |K(x,y)| \: \d (x,y)\right) L^{-\alpha} \, |f|_{\B^{\alpha,p}_q(\R^2)},
\end{equation*}
where
\begin{equation*}
c_{\alpha,q} = \cases{
(2 \e \alpha q)^{\nicefrac{1}{q}} & for $1 \leq q < \infty$, \\
1 & for $q = \infty$.}
\end{equation*}
\end{theorem}

\begin{proof}
First note that due to $f \in \B^{\alpha,p}_q(\R^2) \subset \L^p(\R^2)$ and $K \in \L^1(\R^2)$, we have $K_L \in \L^1(\R^2)$ and
\begin{equation*}
f_L = f * K_L \in \L^p(\R^2)
\quad \forall \, L > 0.
\end{equation*}
Furthermore, $K_L$ and $W$ are related via
\begin{equation*}
\Fourier K_L(x,y) = W\bigg(\frac{\|(x,y)\|_{\R^2}}{L}\bigg)
\quad \mbox{ for } (x,y) \in \R^2
\end{equation*}
so that
\begin{equation*}
\int_{\R^2} K_L(x,y) \: \d (x,y) = \Fourier K_L(0,0) = W(0) = 1.
\end{equation*}
Thus, for $(x,y) \in \R^2$ holds that
\begin{eqnarray*}
f_L(x,y) - f(x,y) & = (f*K_L)(x,y) - f(x,y) \\
& = \int_{\R^2} [f(x-X,y-Y) - f(x,y)] \, K_L(X,Y) \: \d (X,Y).
\end{eqnarray*}
For $p = \infty$ follows that
\begin{eqnarray*}
\fl\|f - f_L\|_{\L^\infty(\R^2)} & \leq \sup_{(x,y) \in \R^2} \int_{\R^2} |f(x-X,y-Y) - f(x,y)| \, |K_L(X,Y)| \: \d (X,Y) \\
& \leq \int_{\R^2} \sup_{(x,y) \in \R^2} |f(x-X,y-Y) - f(x,y)| \, |K_L(X,Y)| \: \d (X,Y).
\end{eqnarray*}
Thus, with the $\L^\infty$-modulus of continuity
\begin{equation*}
\omega_\infty(f,\delta) = \sup_{\|(X,Y)\|_{\R^2} \leq \delta} \sup_{(x,y) \in \R^2} |f(x-X,y-Y) - f(x,y)|
\quad \mbox{ for } \delta > 0
\end{equation*}
we obtain
\begin{equation*}
\|f - f_L\|_{\L^\infty(\R^2)} \leq \int_{\R^2} \omega_\infty(f,\|(X,Y)\|_{\R^2}) \, |K_L(X,Y)| \: \d (X,Y).
\end{equation*}
For $1 \leq p < \infty$ Minkowski's integral inequality gives
\begin{eqnarray*}
\fl\|f - f_L\|_{\L^p(\R^2)} & = \left(\int_{\R^2} \left|\int_{\R^2} [f(x-X,y-Y) - f(x,y)] \, K_L(X,Y) \: \d (X,Y)\right|^p \: \d (x,y)\right)^{\nicefrac{1}{p}} \\
& \leq \int_{\R^2} \left(\int_{\R^2} |f(x-X,y-Y) - f(x,y)|^p \: \d (x,y)\right)^{\nicefrac{1}{p}} |K_L(X,Y)| \: \d(X,Y).
\end{eqnarray*}
Thus, with the $\L^p$-modulus of continuity
\begin{equation*}
\fl\omega_p(f,\delta) = \sup_{\|(X,Y)\|_{\R^2} \leq \delta} \left(\int_{\R^2} |f(x-X,y-Y) - f(x,y)|^p \: \d (x,y)\right)^{\nicefrac{1}{p}}
\quad \mbox{ for } \delta > 0
\end{equation*}
we obtain
\begin{equation*}
\|f - f_L\|_{\L^p(\R^2)} \leq \int_{\R^2} \omega_p(f,\|(X,Y)\|_{\R^2}) \, |K_L(X,Y)| \: \d (X,Y).
\end{equation*}

Consequently, for all $1 \leq p \leq \infty$ we have
\begin{eqnarray*}
\|f - f_L\|_{\L^p(\R^2)} & \leq \int_{\R^2} \omega_p(f,\|(X,Y)\|_{\R^2}) \, |K_L(X,Y)| \: \d (X,Y) \\
& = L^2 \int_{\R^2} \omega_p(f,\|(X,Y)\|_{\R^2}) \, |K(LX,LY)| \: \d (X,Y),
\end{eqnarray*}
where we use the scaling property
\begin{equation*}
K_L(x,y) = L^2 \, K(Lx,Ly)
\quad \forall \, (x,y) \in \R^2.
\end{equation*}
Recall further that the convolution kernel $K$ is radially symmetric, i.e., there exists a univariate function $k:\R \to \R$ such that
\begin{equation*}
K(x,y) = k(\|(x,y)\|_{\R^2})
\quad \forall \, (x,y) \in \R^2.
\end{equation*}
Thus, transforming to polar coordinates gives
\begin{equation*}
\|f - f_L\|_{\L^p(\R^2)} \leq 2\pi \, L^2 \int_0^\infty \omega_p(f,t) \, t \, |k(Lt)| \: \d t.
\end{equation*}
For $q = \infty$ we can conclude that
\begin{eqnarray*}
\|f - f_L\|_{\L^p(\R^2)} & \leq 2\pi \, L^2 \left(\sup_{t>0} t^{-\alpha} \, \omega_p(f,t)\right) \int_0^\infty t^{1+\alpha} \, |k(Lt)| \: \d t \\
& = \left(\int_{\R^2} \|(x,y)\|_{\R^2}^\alpha \, |K(x,y)| \: \d (x,y)\right) L^{-\alpha} \, |f|_{\B^{\alpha,p}_\infty(\R^2)}.
\end{eqnarray*}
Now, let $1 \leq q < \infty$.
Since the $\L^p$-modulus of continuity is monotonically increasing in $\delta > 0$, we can apply Lemma~\ref{lem:embedding_infinity} to obtain
\begin{equation*}
\sup_{t>0} t^{-\alpha} \, \omega_p(f,t) \leq c^{2\alpha} \, \log(c)^{-\nicefrac{1}{q}} \left(\int_0^\infty \big(t^{-\alpha} \, \omega_p(f,t)\Big)^q \: \frac{\d t}{t} \right)^{\nicefrac{1}{q}}
\end{equation*}
for any $c > 1$.
Consequently,
\begin{eqnarray*}
\fl\|f - f_L\|_{\L^p(\R^2)} & \leq 2\pi \, L^2 \left(\sup_{t>0} t^{-\alpha} \, \omega_p(f,t)\right) \int_0^\infty t^{1+\alpha} \, |k(Lt)| \: \d t \\
& \leq c^{2\alpha} \, \log(c)^{-\nicefrac{1}{q}} \left(\int_{\R^2} \|(x,y)\|_{\R^2}^\alpha \, |K(x,y)| \: \d (x,y)\right) L^{-\alpha} \, |f|_{\B^{\alpha,p}_q(\R^2)}.
\end{eqnarray*}
It remains to optimize the constant
\begin{equation*}
C_{\alpha,q}(c) = c^{2\alpha} \, \log(c)^{-\nicefrac{1}{q}}
\quad \mbox{ for } c > 1,
\end{equation*}
which satisfies
\begin{equation*}
C_{\alpha,q}(c) \rightarrow \infty \enspace \mbox{for } c \to 1
\quad \mbox{ and } \quad
C_{\alpha,q}(c) \rightarrow \infty \enspace \mbox{ for } c \to \infty.
\end{equation*}
For $c>1$, we have
\begin{equation*}
C_{\alpha,q}^\prime(c) = \frac{c^{2\alpha-1}\left(2\alpha\log(c) - \frac{1}{q}\right)}{\log(c)^{1+\nicefrac{1}{q}}} = 0
\quad \iff \quad
c = \exp\big((2\alpha q)^{-1}\big)
\end{equation*}
as well as
\begin{eqnarray*}
\fl C_{\alpha,q}^\prime(c) < 0
\quad \forall \, 1 < c < \exp\big((2\alpha q)^{-1}\big),
\qquad C_{\alpha,q}^\prime(c) > 0
\quad \forall \, c > \exp\big((2\alpha q)^{-1}\big).
\end{eqnarray*}
Consequently, the unique minimizer of $C_{\alpha,q}$ on $\R_{>1}$ is given by $c^\ast = \exp\big((2\alpha q)^{-1}\big)$ and
\begin{equation*}
\min_{c>1} C_{\alpha,q}(c) = C_{\alpha,q}(c^\ast) = (2 \e \alpha q)^{\nicefrac{1}{q}}.
\end{equation*}
Hence, for $1 \leq q < \infty$, we have
\begin{equation*}
\fl\|f - f_L\|_{\L^p(\R^2)} \leq (2 \e \alpha q)^{\nicefrac{1}{q}} \left(\int_{\R^2} \|(x,y)\|_{\R^2}^\alpha \, |K(x,y)| \: \d (x,y)\right) L^{-\alpha} \, |f|_{\B^{\alpha,p}_q(\R^2)},
\end{equation*}
which completes the proof.
\end{proof}

Note that for fixed $1 \leq q < \infty$ the constant $c_{\alpha,q}$ in Theorem~\ref{theo:Besov_small} goes to $0$ for $\alpha \searrow 0$ as~$\alpha^{\nicefrac{1}{q}}$.
However, an application of Lemma~\ref{lem:limit} to the $\L^p$-modulus of continuity $\omega_p(f,\cdot)$ shows that in this case the Besov semi-norm $|f|_{\B^{\alpha,p}_q(\R^2)}$ goes to $\infty$ for $\alpha \searrow 0$ as $\alpha^{-\nicefrac{1}{q}}$ and, in particular, we have
\begin{equation*}
\lim_{\alpha \searrow 0} (\alpha q)^{\nicefrac{1}{q}} \, |f|_{\B^{\alpha,p}_q(\R^2)} = \sup_{(X,Y) \in \R^2} \|f(\cdot-X,\cdot-Y) - f\|_{\L^p(\R^2)} \leq 2 \, \|f\|_{\L^p(\R^2)}.
\end{equation*}
Thus, for $\alpha \searrow 0$, the $\L^p$-error estimate in Theorem~\ref{theo:Besov_small} reduces to
\begin{equation*}
\|f - f_L\|_{\L^p(\R^2)} \leq 2(2\e)^{\nicefrac{1}{q}} \|K\|_{\L^1(\R^2)} \, \|f\|_{\L^p(\R^2)}
\end{equation*}
and, for $q \to \infty$, we obtain
\begin{equation*}
\|f - f_L\|_{\L^p(\R^2)} \leq 2 \|K\|_{\L^1(\R^2)} \, \|f\|_{\L^p(\R^2)},
\end{equation*}
which is consistent with simply applying Young's inequality, as for $f \in \L^p(\R^2)$ we have
\begin{equation*}
\|f - f_L\|_{\L^p(\R^2)} \leq \|f\|_{\L^p(\R^2)} + \|K\|_{\L^1(\R^2)} \, \|f\|_{\L^p(\R^2)} \leq 2 \|K\|_{\L^1(\R^2)} \, \|f\|_{\L^p(\R^2)}.
\end{equation*}

\subsection{Error Estimate for $\alpha > 1$}
Convergence results in the general regularization theory for inverse problems typically depend on additional smoothness assumptions on $f$. 
Hence, we now consider the case $\alpha >1$, i.e., functions $f$ which are slightly smoother than sums of characteristic functions.
We assume that $f \in \L^1(\R^2) \cap \B^{\alpha,p}_q(\R^2)$ for $\alpha = n + \theta$ with $n \in \N$, $0 < \theta < 1$ and $1 \leq p,q \leq \infty$.

\begin{theorem}
\label{theo:Besov_large}
Let $f \in \L^1(\R^2) \cap \B^{\alpha,p}_q(\R^2)$ for $\alpha = n + \theta$ with $n \in \N$, $0 < \theta < 1$ and $1 \leq p,q \leq \infty$.
Furthermore, let $W \in \L^\infty(\R)$ be even with $W(0)=1$ such that the corresponding filter $A \equiv A_1$ satisfies $A \in \L^1(\R) \cap \L^2(\R)$ and the convolution kernel $K \equiv K_1$ satisfies $K \in \L^1(\R^2)$ and
\begin{equation*}
\int_{\R^2} x^{j_1} y^{j_2} \, K(x,y) \: \d (x,y) = 0
\quad \forall \, 1 \leq |\boldsymbol{j}| \leq n
\end{equation*}
as well as
\begin{equation*}
\int_{\R^2} \|(x,y)\|_{\R^2}^\alpha \, |K(x,y)| \: \d (x,y) < \infty.
\end{equation*}
Then, the $\L^p$-norm of the inherent FBP reconstruction error $e_L = f - f_L$ is bounded above by
\begin{equation*}
\fl\|e_L\|_{\L^p(\R^2)} \leq c_{\theta,q} \, \frac{\Gamma(\theta+1)}{\Gamma(\alpha+1)} \left(\int_{\R^2} \|(x,y)\|_{\R^2}^\alpha \, |K(x,y)| \: \d (x,y)\right) L^{-\alpha} \, |f|_{\B^{\alpha,p}_q(\R^2)},
\end{equation*}
where
\begin{equation*}
c_{\theta,q} = \cases{
(2 \e \theta q)^{\nicefrac{1}{q}} & for $1 \leq q < \infty$, \\
1 & for $q = \infty$.}
\end{equation*}
\end{theorem}

\begin{proof}
To start with, we recall that due to $f \in \B^{\alpha,p}_q(\R^2) \subset \L^p(\R^2)$ and $K \in \L^1(\R^2)$ we have
\begin{equation*}
f_L = f * K_L \in \L^p(\R^2)
\quad \forall \, L > 0.
\end{equation*}
Furthermore, $K_L$ and $W$ are related via
\begin{equation*}
\Fourier K_L(x,y) = W\bigg(\frac{\|(x,y)\|_{\R^2}}{L}\bigg)
\quad \mbox{ for } (x,y) \in \R^2
\end{equation*}
so that
\begin{equation*}
\int_{\R^2} K_L(X,Y) \: \d (X,Y ) = \Fourier K_L(0,0) = W(0) = 1.
\end{equation*}
Hence, for $(x,y) \in \R^2$ follows that
\begin{eqnarray*}
(f_L - f)(x,y) & = (f*K_L)(x,y) - f(x,y) \\
& = \int_{\R^2} [f(x-X,y-Y) - f(x,y)] \, K_L(X,Y) \: \d (X,Y).
\end{eqnarray*}

To continue, we first assume that $f \in \B^{\alpha,p}_q(\R^2) \cap \Cont^\infty(\R^2)$.
Then, Taylor's theorem gives
\begin{eqnarray*}
\fl f(x-X,y-Y) - f(x,y) \\
= \sum_{1 \leq |\boldsymbol{j}| < n} \frac{(-1)^{|\boldsymbol{j}|}}{\boldsymbol{j}!} \, f^{(\boldsymbol{j})}(x,y) \, X^{j_1} Y^{j_2} \\
\quad + \sum_{|\boldsymbol{j}| = n} \frac{n (-1)^n}{\boldsymbol{j}!} \int_0^1 (1-\tau)^{n-1} \, f^{(\boldsymbol{j})}(x-\tau X,y-\tau Y) \, X^{j_1} Y^{j_2} \: \d \tau
\end{eqnarray*}
and, thus, for $(\dagger) := (f_L - f)(x,y)$ follows that
\begin{eqnarray*}
\fl (\dagger) & = \int_{\R^2} \bigg(\sum_{1 \leq |\boldsymbol{j}| < n} \frac{(-1)^{|\boldsymbol{j}|}}{\boldsymbol{j}!} \, f^{(\boldsymbol{j})}(x,y) \, X^{j_1} Y^{j_2}\bigg) \, K_L(X,Y) \: \d (X,Y) \\
\fl & \quad + \int_{\R^2} \bigg(\sum_{|\boldsymbol{j}| = n} \frac{n (-1)^n}{\boldsymbol{j}!} \int_0^1 (1-\tau)^{n-1} \, f^{(\boldsymbol{j})}_{(\tau X,\tau Y)}(x,y) \, X^{j_1} Y^{j_2} \: \d \tau\bigg) \, K_L(X,Y) \: \d (X,Y),
\end{eqnarray*}
where we set $f^{(\boldsymbol{j})}_{(a,b)}(x,y) = f^{(\boldsymbol{j})}(x-a,y-b)$ for $a,b \in \R$, for the sake of brevity.

By using the assumed moment conditions on $K$ and its scaling property
\begin{equation*}
K_L(x,y) = L^2 \, K(Lx,Ly)
\quad \forall \, (x,y) \in \R^2,
\end{equation*}
we have
\begin{equation*}
\fl\int_{\R^2} X^{j_1} Y^{j_2} \, K_L(X,Y) \: \d (X,Y) = L^{-|\boldsymbol{j}|} \int_{\R^2} x^{j_1} y^{j_2} \, K(x,y) \: \d (x,y) = 0
\quad \forall \, 1 \leq |\boldsymbol{j}| \leq n.
\end{equation*}
Consequently, we obtain
\begin{eqnarray*}
\fl (f_L - f)(x,y)\\
\fl \enspace = \sum_{|\boldsymbol{j}| = n} \frac{n (-1)^n}{\boldsymbol{j}!} \int_0^1 (1-\tau)^{n-1} \int_{\R^2} \bigl(f^{(\boldsymbol{j})}_{(\tau X,\tau Y)} - f^{(\boldsymbol{j})}\bigr)(x,y) \, X^{j_1} Y^{j_2} \, K_L(X,Y) \: \d (X,Y) \, \d \tau \\
\fl \enspace = \sum_{|\boldsymbol{j}| = n} \frac{n (-1)^n}{\boldsymbol{j}!} \int_0^1 (1-\tau)^{n-1} \tau^{-n} \int_{\R^2} \bigl(f^{(\boldsymbol{j})}_{(X,Y)} - f^{(\boldsymbol{j})}\bigr)(x,y) \, X^{j_1} Y^{j_2} \, K_{\frac{L}{\tau}}(X,Y) \: \d (X,Y) \, \d \tau.
\end{eqnarray*}
Thus, for the $\L^p$-norm of the inherent FBP reconstruction error $e_L = f - f_L$ follows that
\begin{eqnarray*}
\fl \|e_L\|_{\L^p(\R^2)} \\
\fl \enspace \leq \sum_{|\boldsymbol{j}| = n} \frac{n}{\boldsymbol{j}!} \int_0^1 (1-\tau)^{n-1} \tau^{-n} \int_{\R^2} \omega_p(f^{(\boldsymbol{j})},\|(X,Y)\|_{\R^2}) \|(X,Y)\|_{\R^2}^n |K_{\frac{L}{\tau}}(X,Y)| \: \d (X,Y) \, \d \tau,
\end{eqnarray*}
where, for $1 \leq p < \infty$, we applied Minkowski's integral inequality as in the proof of Theorem~\ref{theo:Besov_small}.

Recall that the convolution kernel $K$ is radially symmetric, i.e., there exists a univariate function $k:\R \to \R$ such that
\begin{equation*}
K(x,y) = k(\|(x,y)\|_{\R^2})
\quad \forall \, (x,y) \in \R^2.
\end{equation*}
Hence, transforming to polar coordinates gives
\begin{eqnarray*}
\fl \|e_L\|_{\L^p(\R^2)} & \leq 2\pi \, L^2 \sum_{|\boldsymbol{j}| = n} \frac{n}{\boldsymbol{j}!} \int_0^1 (1-\tau)^{n-1} \tau^{-n-2} \int_0^\infty \omega_p(f^{(\boldsymbol{j})},t) \, t^{n+1} \, \Big|k\Big(\frac{L}{\tau}t\Big)\Big| \: \d t \, \d \tau \\
\fl & = 2\pi \, L^2 \sum_{|\boldsymbol{j}| = n} \frac{n}{\boldsymbol{j}!} \int_0^1 (1-\tau)^{n-1} \tau^{-n-2} \int_0^\infty t^{-\theta} \, \omega_p(f^{(\boldsymbol{j})},t) \, t^{\alpha+1} \, \Big|k\Big(\frac{L}{\tau}t\Big)\Big| \: \d t \, \d \tau.
\end{eqnarray*}
For $q = \infty$ we can conclude that
\begin{eqnarray*}
\fl \|e_L\|_{\L^p(\R^2)} \\
\fl \enspace \leq 2\pi \, L^2 \sum_{|\boldsymbol{j}| = n} \frac{n}{\boldsymbol{j}!} \left(\sup_{t>0} t^{-\theta} \, \omega_p(f^{(\boldsymbol{j})},t)\right) \left(\int_0^1 (1-\tau)^{n-1} \tau^{-n-2} \int_0^\infty t^{\alpha+1} \, \Big|k\Big(\frac{L}{\tau}t\Big)\Big| \: \d t \, \d \tau\right) \\
\fl \enspace = 2\pi \, L^{-\alpha} \sum_{|\boldsymbol{j}| = n} \frac{n}{\boldsymbol{j}!} \left(\int_0^1 (1-\tau)^{n-1} \tau^\theta \: \d \tau\right) \left(\int_0^\infty t^{\alpha+1} \, |k(t)| \: \d t\right) |f^{(\boldsymbol{j})}|_{\B^{\theta,p}_\infty(\R^2)} \\
\fl \enspace = L^{-\alpha} \sum_{|\boldsymbol{j}| = n} \frac{n}{\boldsymbol{j}!} \frac{\Gamma(n) \Gamma(\theta+1)}{\Gamma(n+\theta+1)} \left(\int_{\R^2} \|(x,y)\|_{\R^2}^\alpha \, |K(x,y)| \: \d (x,y)\right) |f^{(\boldsymbol{j})}|_{\B^{\theta,p}_\infty(\R^2)}
\end{eqnarray*}
and, thus,
\begin{equation*}
\|e_L\|_{\L^p(\R^2)} \leq \frac{\Gamma(\theta+1)}{\Gamma(\alpha+1)} \left(\int_{\R^2} \|(x,y)\|_{\R^2}^\alpha \, |K(x,y)| \: \d (x,y)\right) L^{-\alpha} \, |f|_{\B^{\alpha,p}_\infty(\R^2)}.
\end{equation*}
For $1 \leq q < \infty$ we can apply Lemma~\ref{lem:embedding_infinity} and obtain, as in the proof of Theorem~\ref{theo:Besov_small},
\begin{equation*}
\fl \|e_L\|_{\L^p(\R^2)} \leq (2 \e \theta q)^{\nicefrac{1}{q}} \, \frac{\Gamma(\theta+1)}{\Gamma(\alpha+1)} \left(\int_{\R^2} \|(x,y)\|_{\R^2}^\alpha \, |K(x,y)| \: \d (x,y)\right) L^{-\alpha} \, |f|_{\B^{\alpha,p}_q(\R^2)}.
\end{equation*}

To prove the result also for functions $f \in \B^{\alpha,p}_q(\R^2)$ that are not smooth, let $\varphi \in \Cont_c^\infty(\R^2)$ be a standard mollifier function, i.e., let $\varphi \geq 0$ satisfy $\supp(\varphi) \subseteq B_1(0)$ and
\begin{equation*}
\int_{\R^2} \varphi(x,y) \: \d (x,y) = 1.
\end{equation*}
Moreover, for $\varepsilon > 0$, we define
\begin{equation*}
\varphi^\varepsilon(x,y) = \varepsilon^{-2} \, \varphi\Bigl(\frac{x}{\varepsilon},\frac{y}{\varepsilon}\Bigr)
\quad \mbox{ for } (x,y) \in \R^2
\end{equation*}
and
\begin{equation*}
f^\varepsilon = f * \varphi^\varepsilon.
\end{equation*}
Then, we have $f^\varepsilon \in \B^{\alpha,p}_q(\R^2) \cap \Cont^\infty(\R^2)$ with $|f^\varepsilon|_{\B^{\alpha,p}_q(\R^2)} \leq |f|_{\B^{\alpha,p}_q(\R^2)}$ and
\begin{equation*}
\fl \|f^\varepsilon\|_{\L^p(\R^2)} \rightarrow \|f\|_{\L^p(\R^2)} \enspace \mbox{ for } \varepsilon \searrow 0
\quad \mbox{ as well as } \quad
|f^\varepsilon|_{\B^{\alpha,p}_q(\R^2)} \rightarrow |f|_{\B^{\alpha,p}_q(\R^2)} \enspace \mbox{ for } \varepsilon \searrow 0.
\end{equation*}
Furthermore, we have already proven that
\begin{equation*}
\fl \|f^\varepsilon - f_L^\varepsilon\|_{\L^p(\R^2)} \leq c_{\theta,q} \, \frac{\Gamma(\theta+1)}{\Gamma(\alpha+1)} \left(\int_{\R^2} \|(x,y)\|_{\R^2}^\alpha \, |K(x,y)| \: \d (x,y)\right) L^{-\alpha} \, |f^\varepsilon|_{\B^{\alpha,p}_q(\R^2)}.
\end{equation*}
For all $L > 0$ we now have
\begin{equation*}
f^\varepsilon - f_L^\varepsilon = f * \varphi^\varepsilon - (f * \varphi^\varepsilon) * K_L = (f - f * K_L) * \varphi^\varepsilon = (f - f_L)^\varepsilon,
\end{equation*}
so that
\begin{equation*}
\fl \|f^\varepsilon - f_L^\varepsilon\|_{\L^p(\R^2)} = \|(f - f_L)^\varepsilon\|_{\L^p(\R^2)} \rightarrow \|f - f_L\|_{\L^p(\R^2)} = \|e_L\|_{\L^p(\R^2)}
\enspace \mbox{ for } \varepsilon \searrow 0.
\end{equation*}
Thus, taking the limit $\varepsilon \searrow 0$ gives
\begin{equation*}
\fl \|e_L\|_{\L^p(\R^2)} \leq c_{\theta,q} \, \frac{\Gamma(\theta+1)}{\Gamma(\alpha+1)} \left(\int_{\R^2} \|(x,y)\|_{\R^2}^\alpha \, |K(x,y)| \: \d (x,y)\right) L^{-\alpha} \, |f|_{\B^{\alpha,p}_q(\R^2)}
\end{equation*}
for any $f \in \L^1(\R^2) \cap \B^{\alpha,p}_q(\R^2)$ and the proof is complete.
\end{proof}

\subsection{Error estimates for noisy data}

We now consider the case of noisy Radon data.
To this end, we assume that the Radon data $\Radon f \in \L^p(\R \times [0,\pi))$, $1 \leq p \leq \infty$, is known only up to a noise level $\delta > 0$ so that we have to reconstruct the target function $f$ from given noisy measurements $g^\delta \in \L^p(\R \times [0,\pi))$ satisfying
\begin{equation*}
\|\Radon f - g^\delta\|_{\L^p(\R \times [0,\pi))} \leq \delta.
\end{equation*}
By applying the approximate FBP formula to the noisy data $g^\delta$, we obtain the reconstruction
\begin{equation*}
f_L^\delta = \frac{1}{2} \, \Back \big(\Fourier^{-1} A_L * g^\delta\big)
\end{equation*}
and the overall FBP reconstruction error $e_L^\delta = f - f_L^\delta$ can be split into an approximation error term and a data error term,
\begin{equation*}
e_L^\delta = f - f_L + f_L - f_L^\delta.
\end{equation*}

In the following, we assume that $f$ is supported in a compact set $\Omega \subset \R^2$ and analyse the $\L^p$-norm of the overall FBP reconstruction error $e_L^\delta$ on $\Omega$ with respect to the noise level $\delta$.
By the triangle inequality, we have
\begin{equation*}
\|e_L^\delta\|_{\L^p(\Omega)} \leq \|f - f_L\|_{\L^p(\Omega)} + \|f_L - f_L^\delta\|_{\L^p(\Omega)}
\end{equation*}
and, consequently, we can treat the approximation error and the data error separately.

The analysis of the data error $f_L - f_L^\delta$ is based on the fact that the back projection $\Back$ defines a mapping
\begin{equation*}
\Back: \L^p(\R \times [0,\pi)) \to \L_\loc^p(\R^2).
\end{equation*}
For $p = \infty$, the definition of $\Back $ reveals that
\begin{equation*}
\|\Back g\|_{\L^\infty(\R^2)} \leq \|g\|_{\L^\infty(\R \times [0,\pi))}
\quad \forall \, g \in \L^\infty(\R \times [0,\pi)).
\end{equation*}
The case $1 \leq p < \infty$ is discussed in the following lemma.

\begin{lemma}
\label{lem:back_projection_L_p}
Let $g \in \L^p(\R \times [0,\pi))$ with $1 \leq p < \infty$.
Then, for any compact subset $\Omega \subset \R^2$ we have
\begin{equation*}
\|\Back g\|_{\L^p(\Omega)} \leq \pi^{-\nicefrac{1}{p}} \, \diam(\Omega)^{\nicefrac{1}{p}} \, \|g\|_{\L^p(\R \times [0,\pi))}
\end{equation*}
and, in particular, $\Back g$ satisfies $\Back g \in \L_\loc^p(\R^2)$.
\end{lemma}

\begin{proof}
Let $g \in \L^p(\R \times [0,\pi))$ with $1 \leq p < \infty$.
For any compact subset $\Omega \subset \R^2$ we have
\begin{eqnarray*}
\|\Back g\|_{\L^p(\Omega)}^p & = \int_\Omega |\Back g(x,y)|^p \: \d (x,y) = \int_{\R^2} |\Back g(x,y)|^p \, \chi_\Omega(x,y) \: \d(x,y) \\
& = \frac{1}{\pi^p} \int_{\R^2} \left|\int_0^\pi g(x \cos(\theta) + y \sin(\theta),\theta) \: \d \theta\right|^p \chi_\Omega(x,y) \: \d(x,y).
\end{eqnarray*}
An application of H\"{o}lder's inequality yields
\begin{equation*}
\fl \|\Back g\|_{\L^p(\Omega)}^p \leq \frac{1}{\pi^p} \int_{\R^2} \left(\pi^{1-\nicefrac{1}{p}} \left(\int_0^\pi |g(x \cos(\theta) + y \sin(\theta),\theta)|^p \: \d \theta\right)^{\nicefrac{1}{p}}\right)^p \chi_\Omega(x,y) \: \d(x,y).
\end{equation*}
By using Fubini's theorem for non-negative functions and the transformation
\begin{equation*}
t = x \cos(\theta) + y \sin(\theta)
\quad \mbox{ and } \quad
s = -x \sin(\theta) + y \cos(\theta),
\end{equation*}
i.e., $\d x \, \d y = \d s \, \d t$ and
\begin{equation*}
x = t \cos(\theta) - s \sin(\theta)
\quad \mbox{ and } \quad
y = t \sin(\theta) + s \cos(\theta),
\end{equation*}
we finally obtain
\begin{eqnarray*}
\fl \|\Back g\|_{\L^p(\Omega)}^p & \leq \frac{1}{\pi} \int_0^\pi \int_\R \int_\R |g(t,\theta)|^p \, \chi_\Omega(t \cos(\theta) - s \sin(\theta),t \sin(\theta) + s \cos(\theta)) \: \d s \, \d t \, \d \theta \\
\fl & = \frac{1}{\pi} \int_0^\pi \int_\R |g(t,\theta)|^p \, \bigg(\int_{\ell_{t,\theta}} \chi_\Omega(x,y) \: \d(x,y)\bigg) \: \d t \, \d \theta \\
\fl & \leq \frac{1}{\pi} \, \diam(\Omega) \, \|g\|_{\L^p(\R \times [0,\pi))}^p < \infty.
\end{eqnarray*}
Consequently, $\Back g$ is defined almost everywhere on $\R^2$ and satisfies $\Back g \in \L_\loc^p(\R^2)$.
\end{proof}

We are now prepared to analyse the data error $f_L - f_L^\delta$ in the $\L^p$-norm for target functions $f \in \L^p(\Omega)$ satisfying $\Radon f \in \L^p(\R \times [0,\pi))$, where
\begin{equation*}
f_L = \frac{1}{2} \, \Back \big(q_L * \Radon f\big)
\quad \mbox{ and } \quad
f_L^\delta = \frac{1}{2} \, \Back \big(q_L * g^\delta\big)
\end{equation*}
with noisy measurements $g^\delta \in \L^p(\R \times [0,\pi))$.

\begin{theorem}[Data error]
\label{theo:data_error}
For compact domain $\Omega \subset \R^2$ and $1 \leq p \leq \infty$ let $f \in \L^p(\Omega)$ satisfy $\Radon f \in \L^p(\R \times [0,\pi))$.
Furthermore, let $W \in \L^\infty(\R)$ be even such that the corresponding filter $A \equiv A_1$ satisfies $A \in \L^1(\R) \cap \L^2(\R)$ as well as $\Fourier^{-1} A \in \L^1(\R)$.
Finally, for $\delta > 0$, let $g^\delta \in \L^p(\R \times [0,\pi))$ be given with
\begin{equation*}
\|\Radon f - g^\delta\|_{\L^p(\R \times [0,\pi))} \leq \delta.
\end{equation*}
Then, the $\L^p$-norm of the data error $f_L - f_L^\delta$ on $\Omega$ is bounded above by
\begin{equation*}
\|f_L - f_L^\delta\|_{\L^p(\Omega)} \leq \frac{1}{2\pi^{\nicefrac{1}{p}}} \, \diam(\Omega)^{\nicefrac{1}{p}} \, \|\Fourier^{-1} A\|_{\L^1(\R)} \, L \, \delta.
\end{equation*}
\end{theorem}

\begin{proof}
We first consider the case $1 \leq p < \infty$.
By the linearity of the back projection $\Back$ and Lemma~\ref{lem:back_projection_L_p} we obtain
\begin{eqnarray*}
\|f_L - f_L^\delta\|_{\L^p(\Omega)} & = \frac{1}{2} \, \|\Back \big(\Fourier^{-1} A_L * (\Radon f - g^\delta)\big)\|_{\L^p(\Omega)} \\
& \leq \frac{1}{2\pi^{\nicefrac{1}{p}}} \, \diam(\Omega)^{\nicefrac{1}{p}} \, \|\Fourier^{-1} A_L * e_\delta\|_{\L^p(\R \times [0,\pi))},
\end{eqnarray*}
where we set $e_\delta = \Radon f - g^\delta$.
By definition of the low-pass filter $A_L$, for $t \in \R$ we have
\begin{eqnarray*}
\fl \Fourier^{-1} A_L(t) & = \frac{1}{2\pi} \int_\R |S| \, W(\nicefrac{S}{L}) \, \e^{\i S t} \: \d S = \frac{L^2}{2\pi} \int_\R |S| \, W(S) \, \e^{\i L S t} \: \d S = L^2 \, \Fourier^{-1} A(Lt)
\end{eqnarray*}
so that
\begin{equation*}
\|\Fourier^{-1} A_L\|_{\L^1(\R)} = L^2 \int_\R |\Fourier^{-1} A(Lt)| \: \d t = L \, \|\Fourier^{-1} A\|_{\L^1(\R)}.
\end{equation*}
Consequently, an application of Young's inequality gives
\begin{eqnarray*}
\fl \|\Fourier^{-1} A_L * e_\delta\|_{\L^p(\R \times [0,\pi))}^p & = \int_0^\pi \|\Fourier^{-1} A_L * e_\delta(\cdot,\theta)\|_{\L^p(\R)}^p \: \d \theta \\
& \leq \int_0^\pi \|\Fourier^{-1} A_L\|_{\L^1(\R)}^p \, \|e_\delta(\cdot,\theta)\|_{\L^p(\R)}^p \: \d \theta \\
& = L^p \, \|\Fourier^{-1} A\|_{\L^1(\R)}^p \, \|e_\delta\|_{\L^p(\R \times [0,\pi)}^p \leq L^p \, \|\Fourier^{-1} A\|_{\L^1(\R)}^p \, \delta^p.
\end{eqnarray*}
By combining the estimates we can conclude that
\begin{equation*}
\|f_L - f_L^\delta\|_{\L^p(\Omega)} \leq \frac{1}{2\pi^{\nicefrac{1}{p}}} \, \diam(\Omega)^{\nicefrac{1}{p}} \, \|\Fourier^{-1} A\|_{\L^1(\R)} \, L \, \delta.
\end{equation*}
Now, let $p = \infty$.
Following along the lines from before, we obtain
\begin{eqnarray*}
\fl \|f_L - f_L^\delta\|_{\L^\infty(\Omega)} & \leq \frac{1}{2} \, \|\Fourier^{-1} A_L * e_\delta\|_{\L^\infty(\R \times [0,\pi))} \leq \frac{1}{2} \, L \, \|\Fourier^{-1} A\|_{\L^1(\R)} \, \|e_\delta\|_{\L^\infty(\R \times [0,\pi))} \\
& \leq \frac{1}{2} \, \|\Fourier^{-1} A\|_{\L^1(\R)} \, L \, \delta
\end{eqnarray*}
and the proof is complete.
\end{proof}

By combining the above result for the data error with our previous findings for the approximation error we can now estimate the $\L^p$-norm of the overall FBP reconstruction error $e_L^\delta = f - f_L^\delta$ via
\begin{equation*}
\|e_L^\delta\|_{\L^p(\Omega)} \leq \|f - f_L\|_{\L^p(\Omega)} + \|f_L - f_L^\delta\|_{\L^p(\Omega)}.
\end{equation*}

\begin{corollary}[Convergence rates for noisy data]
\label{cor:Lp_convergence_rates_noisy}
Let $f \in \L^1(\R^2) \cap \B^{\alpha,p}_q(\R^2)$ for $\alpha = n + \theta$ with $n \in \N$, $0 < \theta < 1$ and $1 \leq p,q \leq \infty$ be supported in a compact domain $\Omega \subset \R^2$.
Furthermore, let $W \in \L^\infty(\R)$ be even with $W(0)=1$ such that the corresponding filter $A \equiv A_1$ satisfies $A \in \L^1(\R) \cap \L^2(\R)$ as well as $\Fourier^{-1} A \in \L^1(\R)$ and the convolution kernel $K \equiv K_1$ satisfies $K \in \L^1(\R^2)$ and
\begin{equation*}
\int_{\R^2} x^{j_1} y^{j_2} \, K(x,y) \: \d (x,y) = 0
\quad \forall \, 1 \leq |\boldsymbol{j}| \leq n
\end{equation*}
as well as
\begin{equation*}
\int_{\R^2} \|(x,y)\|_{\R^2}^\alpha \, |K(x,y)| \: \d (x,y) < \infty.
\end{equation*}
Finally, let $g^\delta \in \L^p(\R \times [0,\pi))$ be given with
\begin{equation*}
\|\Radon f - g^\delta\|_{\L^p(\R \times [0,\pi))} \leq \delta.
\end{equation*}
Then, the $\L^p$-norm of the overall FBP reconstruction error $e_L^\delta = f - f_L^\delta$ on $\Omega$ is bounded above by
\begin{equation*}
\|f - f_L^\delta\|_{\L^p(\Omega)} \leq (c_{W,\alpha,q} + c_{W,\Omega,p}) \, |f|_{\B^{\alpha,p}_q(\R^2)}^{\frac{1}{\alpha+1}} \, \delta^{\frac{\alpha}{\alpha+1}},
\end{equation*}
where the bandwidth $L$ is chosen as
\begin{equation*}
L = \delta^{-\frac{1}{\alpha+1}} \, |f|_{\B^{\alpha,p}_q(\R^2)}^{\frac{1}{\alpha+1}}
\end{equation*}
and the involved constants are given by
\begin{equation*}
c_{W,\Omega,p} = \frac{1}{2\pi^{\nicefrac{1}{p}}} \, \diam(\Omega)^{\nicefrac{1}{p}} \, \|\Fourier^{-1} A\|_{\L^1(\R)}
\end{equation*}
and
\begin{equation*}
c_{W,\alpha,q} = c_{\theta,q} \, \frac{\Gamma(\theta+1)}{\Gamma(\alpha+1)} \left(\int_{\R^2} \|(x,y)\|_{\R^2}^\alpha \, |K(x,y)| \: \d (x,y)\right)
\end{equation*}
with
\begin{equation*}
c_{\theta,q} = \cases{
(2 \e \theta q)^{\nicefrac{1}{q}} & for $1 \leq q < \infty$, \\
1 & for $q = \infty$.}
\end{equation*}
In particular,
\begin{equation*}
\|f - f_L^\delta\|_{\L^p(\Omega)} = \Or\bigl(\delta^{\frac{\alpha}{\alpha+1}}\bigr)
\quad \mbox{ for } \quad
\delta \searrow 0.
\end{equation*}
\end{corollary}

\begin{proof}
According to Theorem~\ref{theo:Besov_large} the $\L^p$-norm of the approximation error $f - f_L$ is bounded above by
\begin{equation*}
\|f - f_L\|_{\L^p(\Omega)} \leq \|f - f_L\|_{\L^p(\R^2)} \leq c_{W,\alpha,q} \, L^{-\alpha} \, |f|_{\B^{\alpha,p}_q(\R^2)}
\end{equation*}
and by Theorem~\ref{theo:data_error} the $\L^p$-norm of the data error $f_L - f_L^\delta$ can be estimated in terms of the noise level $\delta$ via
\begin{equation*}
\|f_L - f_L^\delta\|_{\L^p(\Omega)} \leq c_{W,\Omega,p} \, L \, \delta.
\end{equation*}

Thus, the $\L^p$-norm of the overall FBP reconstruction error $f - f_L^\delta$ can be bounded above by
\begin{equation*}
\|f - f_L^\delta\|_{\L^p(\Omega)} \leq c_{W,\alpha,q} \, L^{-\alpha} \, |f|_{\B^{\alpha,p}_q(\R^2)} + c_{W,\Omega,p} \, L \, \delta
\end{equation*}
and choosing the bandwidth $L$ as
\begin{equation*}
L = \delta^{-\frac{1}{\alpha+1}} \, |f|_{\B^{\alpha,p}_q(\R^2)}^{\frac{1}{\alpha+1}},
\end{equation*}
we obtain
\begin{equation*}
\|f - f_L^\delta\|_{\L^p(\Omega)} \leq (c_{W,\alpha,q} + c_{W,\Omega,p}) \, |f|_{\B^{\alpha,p}_q(\R^2)}^{\frac{1}{\alpha+1}} \, \delta^{\frac{\alpha}{\alpha+1}},
\end{equation*}
as stated.
\end{proof}

Note that the decay rate of the $\L^p$-error bound in Corollary~\ref{cor:Lp_convergence_rates_noisy} is independent of $1 \leq p \leq \infty$, as
\begin{equation*}
\|f - f_L^\delta\|_{\L^p(\Omega)} = \Or\bigl(\delta^{\frac{\alpha}{\alpha + 1}}\bigr)
\quad \mbox{ for } \quad
\delta \searrow 0,
\end{equation*}
where the filter's bandwidth $L > 0$ has to go to $\infty$ as the noise level $\delta > 0$ goes to $0$ with rate
\begin{equation*}
L = \Or\bigl(\delta^{-\frac{1}{\alpha+1}}\bigr)
\quad \mbox{ for } \quad
\delta \searrow 0.
\end{equation*}

To close this section we give an example of a filter function $A_{L,\nu}$ depending on a parameter $\nu\in\N_0$, which fulfils the assumptions of Theorems~\ref{theo:Besov_small},~\ref{theo:Besov_large}~and~\ref{theo:data_error} for suitable $\nu$.
We remark that these assumptions especially imply continuity of the window $W$ on $\R$ so that they are not satisfied for the typical choices summarized in Table~\ref{tab:lpf}.

\begin{example}
\label{ex:smooth_filter}
We define the {\em smooth filter of order} $\nu \in \N_0$ as $A_{L,\nu} = |\cdot| \, W_\nu(\nicefrac{\cdot}{L})$ with
\begin{equation*} 
W_{\nu}(S) = \cases{ (1-|S|^2)^{\nu} & for $|S| \leq 1$, \\
0 & for $|S|\, > 1$. } 
\end{equation*}
It is easy to verify that $W_{\nu}\in \L^\infty (\R)$ is an even function with $W(0) = 1$ and $A_{L, \nu} \in \L^1(\R) \cap \L^2(\R)$ for all $\nu \in \N_0$ and $L>0$.
In the following, we analyse the inverse Fourier transform of $A_{L,\nu}$ and the corresponding convolution kernel $K_{L,\nu}$.

The inverse Fourier transform of $A_{L,\nu}$ involves the generalized hypergeometric function ${}_1F_2$ and is given by
\begin{equation*}
\Fourier^{-1} A_{L,\nu}(s) = \frac{L^2}{2\pi}\, \B(\nu+1, 1)\, {}_1F_2(1;1/2,\nu+2;\nicefrac{-L^2 s^2}{4}) \quad \mbox{ for } s\in\R.
\end{equation*}
It can be verified that $\Fourier^{-1} A_{L,\nu} \in \L^1(\R)$ for all $\nu \in \N$, but $\Fourier^{-1} A_{L,\nu} \not\in \L^1(\R)$ for $\nu = 0$.
The corresponding convolution kernel $K_{L,\nu}$ can be computed as
\begin{equation*}
K_{L,\nu}(x,y) = \cases{ \frac{L^2}{4\pi\, (\nu +1)} & for $\|(x,y)\|_{\R^2} = 0$, \\ 
\frac{L^2}{2\pi} \, 2^\nu \, \Gamma (\nu+1)\,  \frac{J_{\nu+1}(L \, \|(x,y)\|_{\R^2})}{(L \, \|(x,y)\|_{\R^2})^{\nu +1}} & for $\|(x,y)\|_{\R^2} > 0$,}
\end{equation*}
where $J_{\nu+1}$ denotes the Bessel function of the first kind of order $\nu+1$ defined by
\begin{equation*}
J_{\nu+1}(t) = \frac{1}{\pi} \int_0^\pi \cos(t\sin(\varphi) - (\nu+1)\, \varphi) \:\d\varphi 
\quad \mbox{ for } t \in \R. 
\end{equation*}
For fixed $\alpha\geq 0$ we now develop a condition on $\nu \in \N_0$ such that the integral 
\begin{equation*}
I_{\alpha,\nu} = \int_{\R^2} \|(x,y)\|_{\R^2}^\alpha \, |K_{\nu}(x,y)| \: \d (x,y)
\end{equation*}
is finite, where we set $K_\nu \equiv K_{1,\nu}$ for the sake of brevity.
To this end, first observe that 
\begin{equation*}
I_{\alpha,\nu} = 2^{\nu} \, \Gamma (\nu+1)\, \int_0^\infty \frac{|J_{\nu+1}(r)|}{r^{\nu - \alpha}} \: \d r.
\end{equation*}
Since $|J_{\nu+1}(r)| \, r^{\alpha-\nu}$ is bounded on $[0,\eta]$ for all $\eta > 0$ due to \cite[9.1.7 \& 9.1.60]{Abramowitz1972}, we have
\begin{equation*}
\int_{\R^2} \|(x,y)\|_{\R^2}^\alpha \, |K_{\nu}(x,y)| \: \d (x,y) 
\leq c_\eta + 2^{\nu} \, \Gamma (\nu+1)\, \int_\eta^\infty \frac{|J_{\nu+1}(r)|}{r^{\nu - \alpha}} \: \d r.
\end{equation*} 
According to~\cite[9.2.1]{Abramowitz1972}, the Bessel function $J_{\nu+1}$ may be written as
\begin{equation*}
J_{\nu+1}(t)= \sqrt{\frac{2}{\pi t}} \left(\cos(t-\nicefrac{1}{2}\,\pi\,\nu-\nicefrac{3}{4}\,\pi)+\mathcal{O}(|t|^{-1})\right)
\quad \mbox{ for } \quad
|t| \to \infty.
\end{equation*}
Therefore, choosing $\eta$ sufficiently large yields 
\begin{equation*}
\fl \int_{\R^2} \|(x,y)\|_{\R^2}^\alpha \, |K_{\nu}(x,y)| \: \d (x,y) 
\leq C_\eta + \frac{2^{\nu+\nicefrac{1}{2}}}{\sqrt{\pi}} \,\Gamma (\nu+1)\, \int_\eta^\infty \frac{|\cos(r-\nicefrac{1}{2}\,\pi\,\nu-\nicefrac{3}{4}\,\pi)|}{r^{\nu - \alpha + \nicefrac{1}{2}}} \: \d r
\end{equation*}
and the latter integral converges for $\nu > \alpha + \nicefrac{1}{2}$.

We now prove divergence of $I_{\alpha,\nu}$ for $\nu \leq \alpha + \nicefrac{1}{2}$.
For sufficiently large $N \in \N$ we have
\begin{equation*}
\fl \int_{\R^2} \|(x,y)\|_{\R^2}^\alpha \, |K_{\nu}(x,y)| \: \d (x,y) \geq \frac{2^{\nu+\nicefrac{1}{2}}}{\sqrt{\pi}} \,\Gamma (\nu+1)\, \int_{\pi N}^\infty \frac{\cos^2(r-\nicefrac{1}{2}\,\pi\,\nu-\nicefrac{3}{4}\,\pi)}{r^{\nu - \alpha + \nicefrac{1}{2}}} \: \d r.
\end{equation*}
If $\alpha - \nicefrac{1}{2} < \nu \leq \alpha + \nicefrac{1}{2}$, we obtain 
\begin{equation*}
\fl \int_{\R^2} \|(x,y)\|_{\R^2}^\alpha \, |K_{\nu}(x,y)| \: \d (x,y) 
\geq 2^{\nu-\nicefrac{1}{2}}\,\sqrt{\pi} \, \Gamma (\nu+1)\,  \sum_{n=N}^\infty ((n+1)\pi)^{-\nu+\alpha-\nicefrac{1}{2}},
\end{equation*}
which diverges as we have $\nu \leq \alpha + \nicefrac{1}{2} \: \Leftrightarrow \: -\nu+\alpha-\nicefrac{1}{2} \geq -1$.
If $\nu \leq \alpha - \nicefrac{1}{2}$, we obtain
\begin{equation*}
\int_{\R^2} \|(x,y)\|_{\R^2}^\alpha \, |K_{\nu}(x,y)| \: \d (x,y) 
\geq 2^{\nu-\nicefrac{1}{2}}\,\sqrt{\pi} \, \Gamma (\nu+1)\,  \sum\limits_{n=N}^\infty (n\,\pi)^{-\nu+\alpha-\nicefrac{1}{2}},
\end{equation*}
which diverges as well because $\nu \leq \alpha - \nicefrac{1}{2} \: \Leftrightarrow \: -\nu+\alpha-\nicefrac{1}{2} \geq 0$.
In summary, we have
\begin{equation*}
\int_{\R^2} \|(x,y)\|_{\R^2}^\alpha \, |K_{\nu}(x,y)| \: \d (x,y) < \infty
\quad \iff \quad
\nu > \alpha + \nicefrac{1}{2}.
\end{equation*}
In particular, we have proven that $K_\nu \in \L^1(\R^2)$ if and only if $\nu > \nicefrac{1}{2}$.
It remains to determine the maximal $n \in \N$ such that
\begin{equation*}
\int_{\R^2} x^{j_1} y^{j_2} \, K_\nu(x,y) \: \d (x,y) = 0
\quad \forall \, 1 \leq |\boldsymbol{j}| \leq n.
\end{equation*}
By transforming to polar coordinates the above integral can be rewritten as 
\begin{equation*}
\fl \int_{\R^2} x^{j_1} y^{j_2} \, K_\nu(x,y) \: \d (x,y) = \frac{1}{2\pi} \, 2^\nu \,  \Gamma (\nu+1)\int_0^{2\pi} \cos^{j_1}(\varphi)\, \sin^{j_2}(\varphi) \: \d\varphi \, \int_0^\infty  \frac{J_{\nu+1}(r)}{r^{\nu - |\boldsymbol{j}|}} \:\d r.
\end{equation*}
For $\nu > n - \nicefrac{1}{2}$ we have
\begin{equation*}
\int_0^\infty \frac{J_{\nu+1}(r)}{r^{\nu - |\boldsymbol{j}|}} \:\d r \neq 0
\quad \forall \, 1 \leq |\boldsymbol{j}| \leq n
\end{equation*}
and, moreover,
\begin{equation*}
\int_0^{2\pi} \cos^{j_1}(\varphi)\, \sin^{j_2}(\varphi)\:\d \varphi = 0
\quad \forall \, 1 \leq |\boldsymbol{j}| \leq n
\quad \iff \quad
n = 1.
\end{equation*}
Consequently, the smooth filter of order $\nu \in \N$ satisfies the moment conditions
\begin{equation*}
\int_{\R^2} x^{j_1} y^{j_2} \, K_\nu(x,y) \: \d (x,y) = 0
\quad \forall \, 1 \leq |\boldsymbol{j}| \leq n.
\end{equation*}
only for $n = 1$.
\end{example}

Summarizing the results of Example~\ref{ex:smooth_filter}, the smooth filter of order $\nu \in \N_0$ satisfies the assumptions of our error theory in Theorems~\ref{theo:Besov_small},~\ref{theo:Besov_large}~and~\ref{theo:data_error} for all $\alpha<2$ iff $\nu > \alpha + \nicefrac{1}{2}$.

\section{Numerical experiments}

We now present selected numerical examples to illustrate our theoretical results.
To this end, we assume that the target function $f$ is compactly supported with
\begin{equation*}
\supp(f) \subseteq B_1(0) = \bigl\{(x,y) \in \R^2 \bigm| x^2+y^2 \leq 1\bigr\}
\end{equation*}
and that the Radon data are given in {\em parallel beam geometry}
\begin{equation*}
\big\{(\Radon f)_{m,n} = \Radon f(m \, d,n \, \nicefrac{\pi}{N}) \mid -M \leq m \leq M, ~ 0 \leq n \leq N-1\big\},
\end{equation*}
where $d$ is the spacing of $2M+1$ parallel lines per angle and $N$ is the number of angles.

The reconstruction of $f$ from discrete Radon data requires a suitable {\em discretization} of the approximate FBP reconstruction formula
\begin{equation*}
f_L = \frac{1}{2} \Back\bigl(\Fourier^{-1} A_L * \Radon f\bigr).
\end{equation*}
We follow a standard approach~\cite{Natterer2001a} and apply the composite trapezoidal rule to discretize the convolution $\ast$ and back projection $\Back$, leading to the discrete convolution $*_D$ and discrete back projection $\Back_D$, respectively.
Moreover, we apply an interpolation method~$\Int$ to reduce the computational costs.
This yields the {\em discrete FBP reconstruction formula}
\begin{equation*}
f_\FBP = \frac{1}{2} \Back_D \bigl(\Int[\Fourier^{-1} A_L *_D \Radon f]\bigr).
\end{equation*}
For target functions $f$ of low regularity it is sufficient to use linear spline interpolation.
To exploit a higher regularity we apply cubic spline interpolation.
Furthermore, we couple the parameters $d > 0$ and $M,N \in \N$ with the bandwidth $L$ via
\begin{equation*}
d = \frac{\pi}{L}, \quad M = \frac{1}{d}, \quad N = \lceil \pi \, M \rceil
\end{equation*}
and choose $L$ to be a multiple of $\pi$, i.e., $L = k\pi$ for some  $k \in \N$.

\medskip

In our numerical experiments, we use the {\em Shepp-Logan phantom} with attenuation function
\begin{equation*}
f_{\mathrm{SL}} = \sum_{j=1}^{10} c_j \, f_j,
\end{equation*}
where each function $f_j$ is of the form of the characteristic function of an ellipse given by
\begin{equation*}
f_e(x,y) = \chi_{B_1(0)}\big(\bfx_{a,b,h,k,\varphi}(x,y)\big)
\end{equation*}
with
\begin{equation*}
\fl \bfx_{a,b,h,k,\varphi}(x,y) = \left(\frac{(x-h)\cos(\varphi) + (y-k)\sin(\varphi)}{a},\frac{-(x-h)\sin(\varphi) + (y-k)\cos(\varphi)}{b}\right).
\end{equation*}
The parameters of the ellipses used in the Shepp-Logan phantom can be found in~\cite{Shepp1974}.
For illustration, the Shepp-Logan phantom and its sinogram are shown in Figure~\ref{fig:shepp_logan_phantom}.

\begin{figure}[t]
\centering
\includegraphics{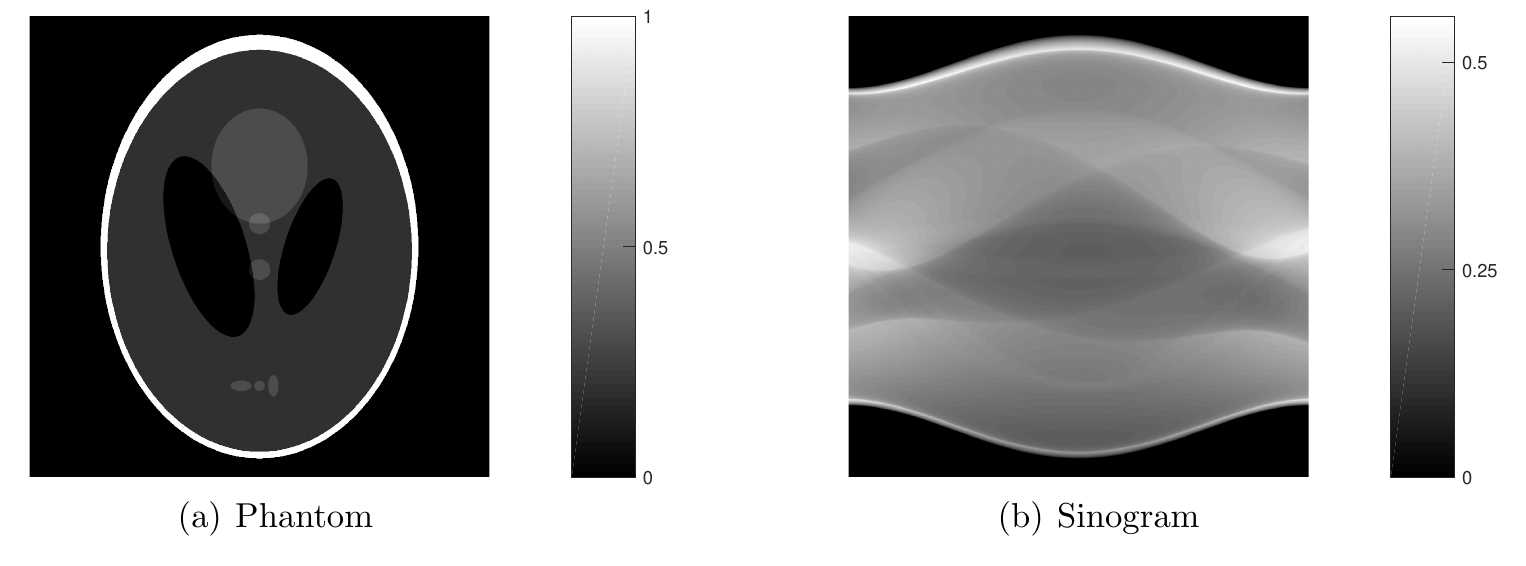}
\caption{The Shepp-Logan phantom and its sinogram.}
\label{fig:shepp_logan_phantom}
\end{figure}

According to~\cite{Berkner2011}, the function $f_\mathrm{SL}$ belongs to the Besov space $\B^{\alpha,p}_p(\R^2)$ for~${\alpha < \nicefrac{1}{p}}$, which determines the decay rate of the FBP approximation error $e_L = f - f_L$ in the $\L^p$-norm according to Theorem~\ref{theo:Besov_small}.
To observe higher rates of convergence we consider the function
\begin{equation*}
p_\sigma(x,y) = \cases{
(1-x^2-y^2)^\sigma & for $x^2+y^2 \leq 1$, \\
0 & for $x^2+y^2 > 1$}
\end{equation*}
with parameter $\sigma > 0$, which is in $\B^{\alpha,p}_p(\R^2)$ for $\alpha < \sigma + \nicefrac{1}{p}$.
Adapting the approach in~\cite{Rieder2003}, we then define the {\em smooth phantom of order $\sigma$} via
\begin{equation*}
f_{\mathrm{smooth}}^\sigma = f_1^\sigma - \frac{3}{2} \, f_2^\sigma + \frac{3}{2} \, f_3^\sigma \in \B^{\alpha,p}_p(\R^2)
\quad \forall \, \alpha < \sigma + \frac{1}{p},
\end{equation*}
where each function $f_j^\sigma$ is of the form
\begin{equation*}
f_\sigma(x,y) = p_\sigma\big(\bfx_{a,b,h,k,\varphi}(x,y)\big).
\end{equation*}
The parameters used in the definition of the smooth phantom can be found in~\cite{Rieder2003}.
For illustration, Figure~\ref{fig:smooth1_phantom} shows the smooth phantom of order $\sigma = 1$ along with its sinogram.

\begin{figure}[b]
\centering
\includegraphics{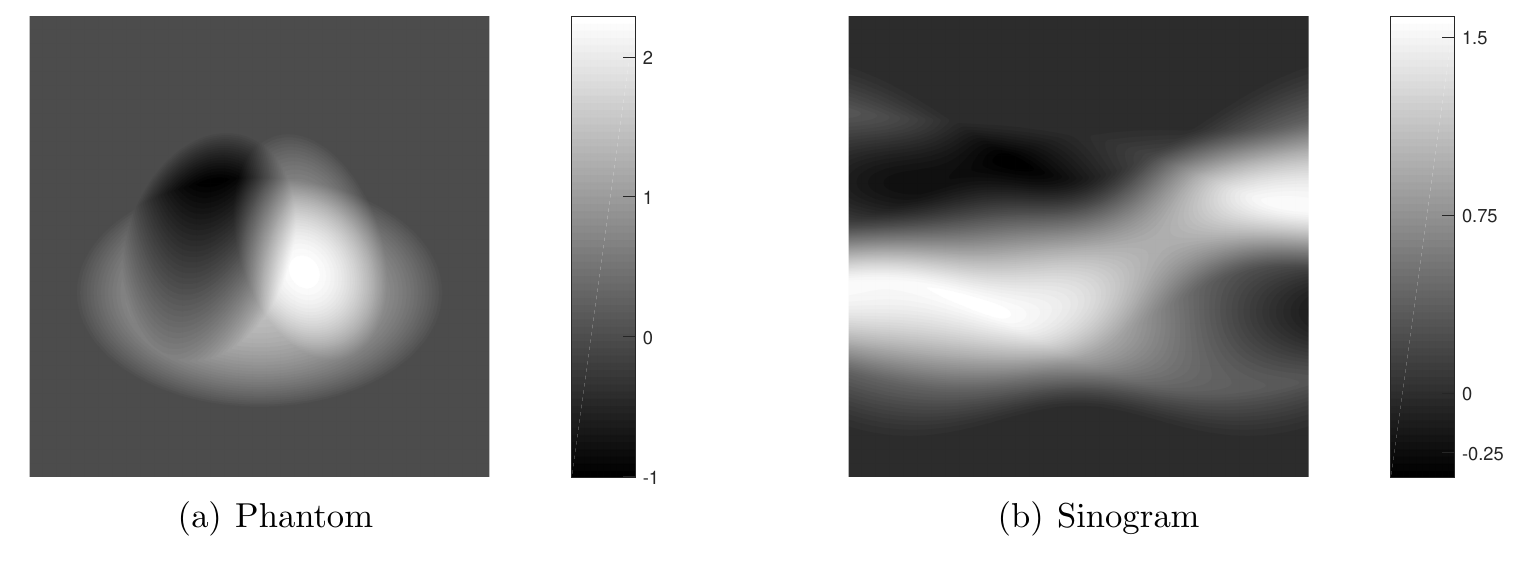}
\caption{The smooth phantom of order $\sigma = 1$ and its sinogram.}
\label{fig:smooth1_phantom}
\end{figure}

The FBP reconstructions of both phantoms are displayed in Figure~\ref{fig:FBP_reconstruction}, where we use the smooth filter from Example~\ref{ex:smooth_filter}, i.e.,
\begin{equation*}
A_L(S) = \cases{
|S| \, \bigl(1-L^{-2} \, S^2\bigr)^\nu & for $|S| \leq L$, \\
0 & for $|S| > L$,}
\end{equation*}
with $\nu = 5$ and $L = 100\pi$.
This corresponds to $M = 100$ and $N = 315$ so that $(2M+1)N = 63315$ Radon samples are taken.
In our numerical experiments, we evaluate the phantoms and reconstructions on a square grid with $1024 \times 1024$ pixels.

\begin{figure}[t]
\centering
\includegraphics{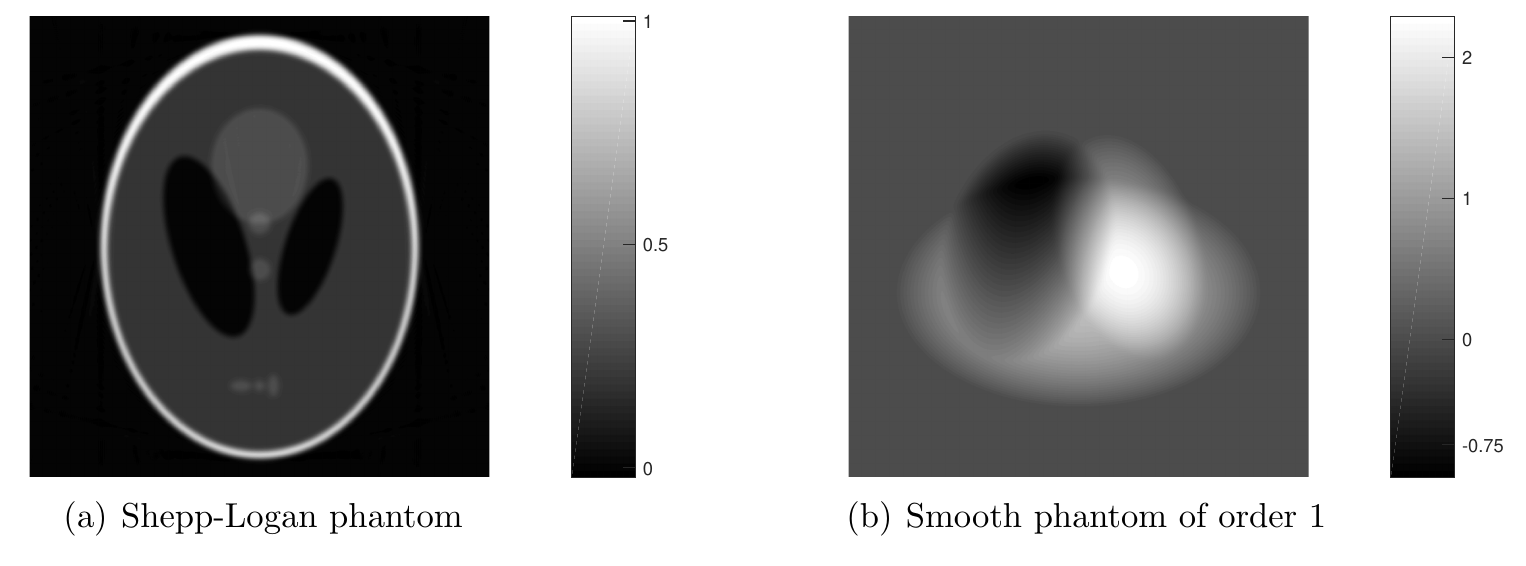}
\caption{FBP reconstructions with smooth filter of order $\nu = 5$ and $L = 100\pi$.}
\label{fig:FBP_reconstruction}
\end{figure}

We start with illustrating our theoretical results concerning the approximation error
\begin{equation*}
e_L = f - f_L.
\end{equation*}
To this end, Figure~\ref{fig:approximation_error_shepp-logan} shows the discrete $\L^p$-norm of the FBP approximation error of the Shepp-Logan phantom for $p \in \{1,\nicefrac{4}{3},2,4\}$ as a function of the bandwidth $L$ in logarithmic scales for the smooth filter with $\nu \in \{5,7\}$.
In this case, Theorem~\ref{theo:Besov_small} gives
\begin{equation*}
\|e_L\|_{\L^p(\R^2)} \leq (2 \e \alpha p)^{\nicefrac{1}{p}} \left(\int_{\R^2} \|(x,y)\|_{\R^2}^\alpha \, |K(x,y)| \: \d (x,y)\right) L^{-\alpha} \, |f|_{\B^{\alpha,p}_p(\R^2)}.
\end{equation*}

\begin{figure}[p]
\centering
\includegraphics{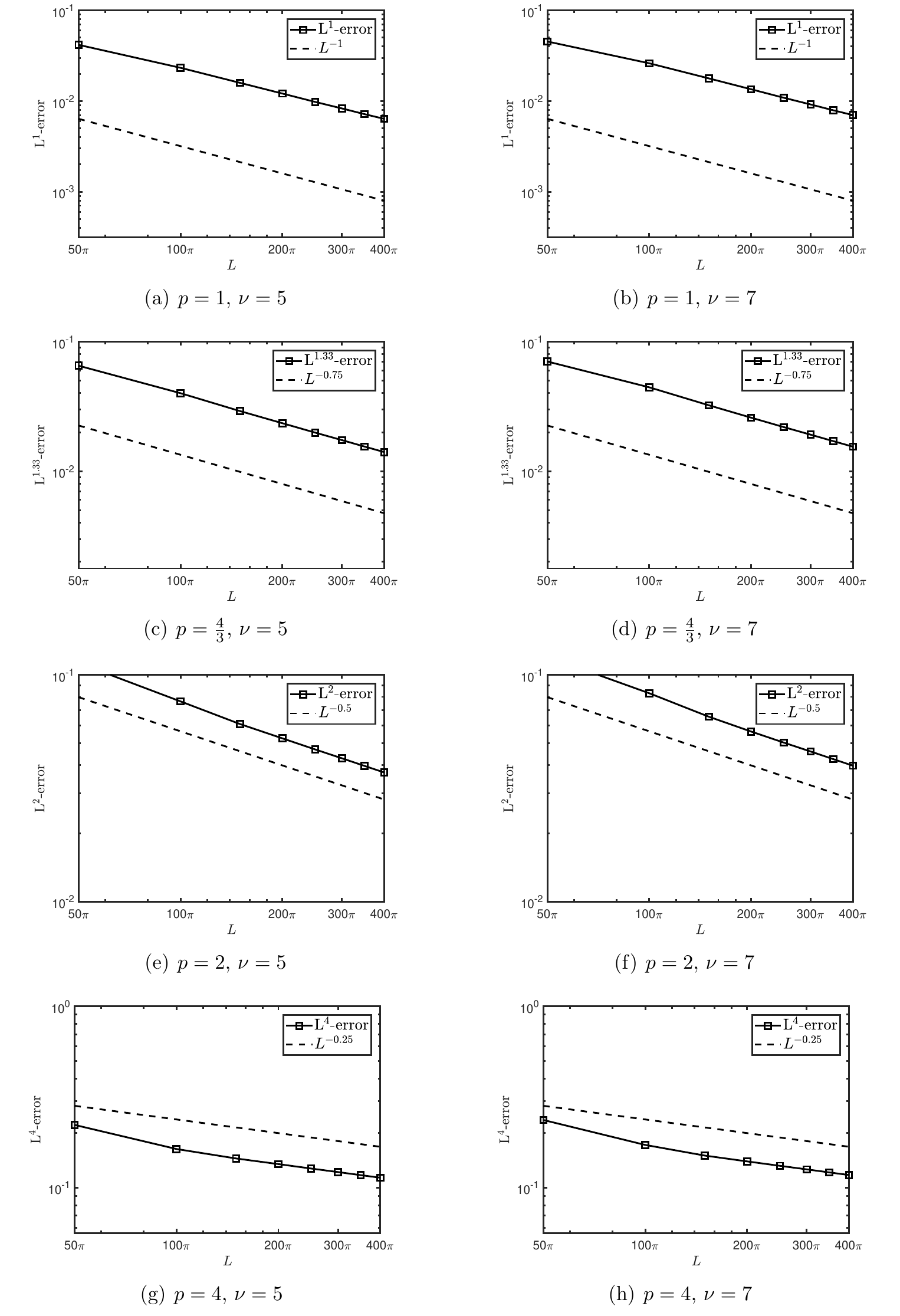}
\caption{$\L^p$-approximation error for Shepp-Logan phantom and smooth filter.}
\label{fig:approximation_error_shepp-logan}
\end{figure}

\begin{figure}[p]
\centering
\includegraphics{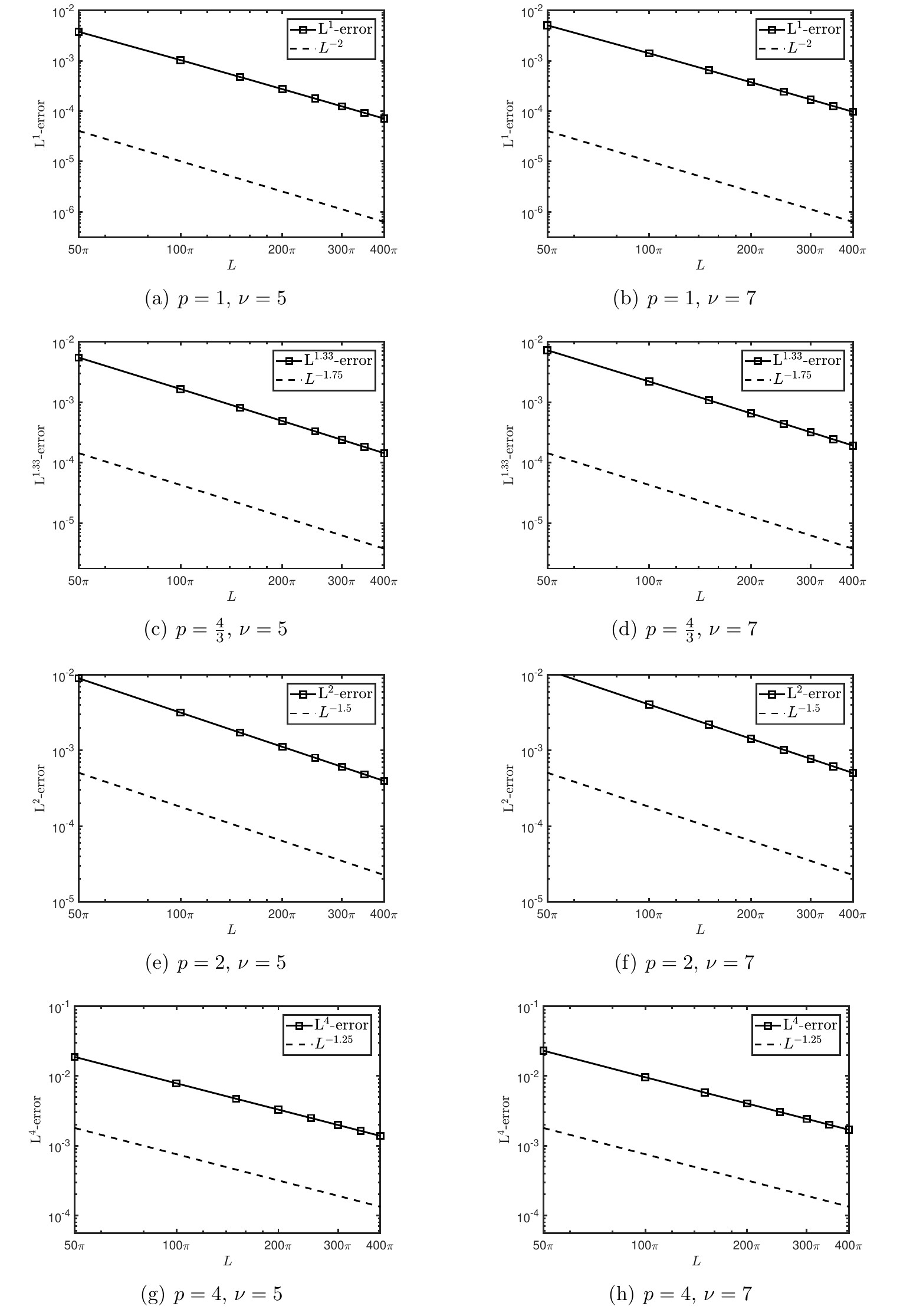}
\caption{$\L^p$-approximation error for smooth phantom ($\sigma = 1$) and smooth filter.}
\label{fig:approximation_error_smooth1}
\end{figure}

\begin{figure}[t]
\centering
\includegraphics{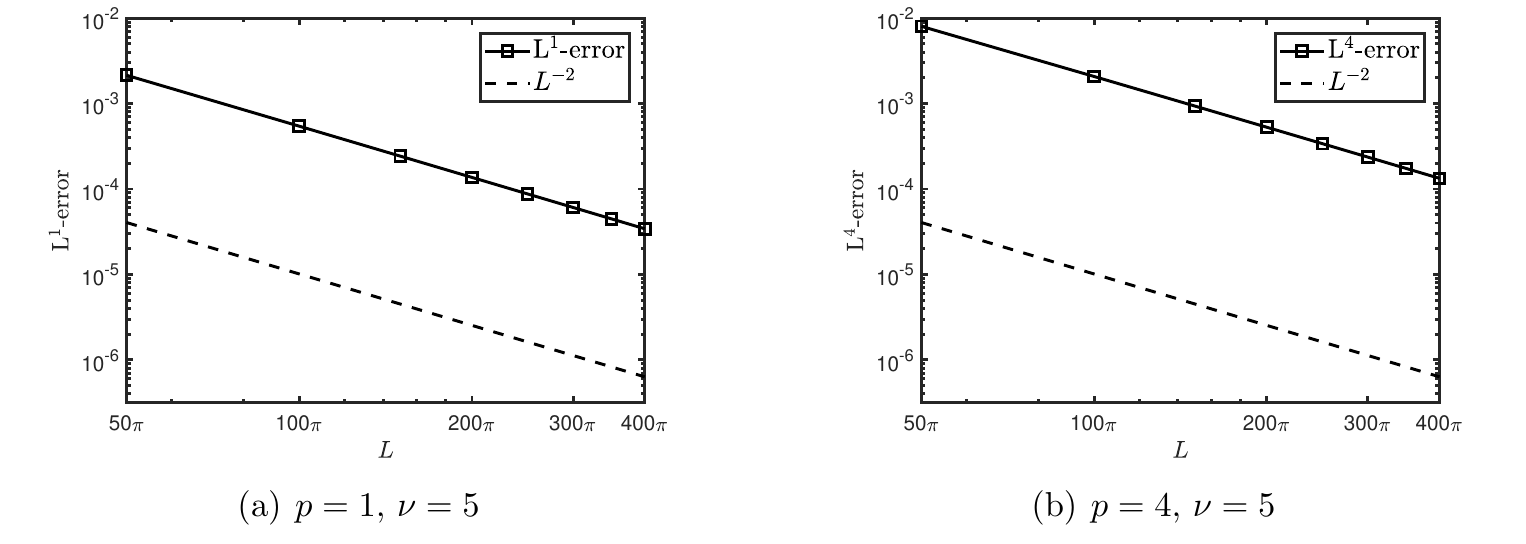}
\caption{$\L^p$-approximation error for smooth phantom ($\sigma = 2$) and smooth filter.}
\label{fig:approximation_error_smooth2}
\end{figure}

In all cases, the plots in Figure~\ref{fig:approximation_error_shepp-logan} show that the discrete $\L^p$-norm of the FBP approximation error decreases with increasing bandwidth $L$ with rate $L^{-\nicefrac{1}{p}}$ for both $\nu = 5$ and $\nu = 7$.
This is exactly the behaviour we expect as we have $f_{\mathrm{SL}} \in \B^{\alpha,p}_p(\R^2)$ for any $\alpha < \nicefrac{1}{p}$.
Moreover, we see that the error is smaller for $\nu = 5$ than for $\nu = 7$.
This observation also fits to our expectations because the constant
\begin{equation*}
c_{\alpha,K} = \int_{\R^2} \|(x,y)\|_{\R^2}^\alpha \, |K(x,y)| \: \d(x,y)
\end{equation*}
is smaller for $\nu = 5$ for all corresponding values of $\alpha \in \{1,\nicefrac{3}{4},\nicefrac{1}{2},\nicefrac{1}{4}\}$, see Table~\ref{tab:c_alpha_K}.
Note that this behaviour of the error can also be observed for other choices of $1 \leq p < \infty$ and the smooth filter with alternative parameter $\nu \in \N$.

\begin{table}[b]
\centering
\begin{tabular}{c|cccccccc}
\backslashbox{$\nu$}{$\alpha$} & $\nicefrac{1}{4}$ & $\nicefrac{1}{2}$ & $\nicefrac{3}{4}$ & $1$  & $\nicefrac{5}{4}$ & $\nicefrac{3}{2}$ & $\nicefrac{7}{4}$ & $2$ \\
\hline\\[-1ex]
$5$ & $1.4273$ & $2.0329$ & $2.9484$ & $4.3460$ & $6.5018$ & $9,8643$ & $15.1708$ & $23.6530$ \\[0.5ex]
$7$ & $1.4538$ & $2.1409$ & $3.2078$ & $4.8797$ & $7.5234$ & $11.7401$ & $18.5234$ & $29.5256$
\end{tabular}
\caption{Numerical approximation of $c_{\alpha,K}$ for smooth filter of order $\nu \in \{5,7\}$.}
\label{tab:c_alpha_K}
\end{table}

Figure~\ref{fig:approximation_error_smooth1} now shows the discrete $\L^p$-norm of the FBP approximation error of the smooth phantom of order $\sigma = 1$ for $p \in \{1,\nicefrac{4}{3},2,4\}$ as a function of the bandwidth $L$ in logarithmic scales for the smooth filter with $\nu \in \{5,7\}$.
In this case, Theorem~\ref{theo:Besov_large} gives
\begin{equation*}
\fl\|e_L\|_{\L^p(\R^2)} \leq (2 \e \theta p)^{\nicefrac{1}{p}} \, \frac{\Gamma(\theta+1)}{\Gamma(\alpha+1)} \left(\int_{\R^2} \|(x,y)\|_{\R^2}^\alpha \, |K(x,y)| \: \d (x,y)\right) L^{-\alpha} \, |f|_{\B^{\alpha,p}_p(\R^2)},
\end{equation*}
where $\theta = \alpha - \lfloor\alpha\rfloor$ is the fractional part of $\alpha > 1$.

In all cases, the plots in Figure~\ref{fig:approximation_error_smooth1} show that the discrete $\L^p$-norm of the approximation error decreases as $L^{-(1+\nicefrac{1}{p})}$ for both $\nu = 5$ and $\nu = 7$.
This is exactly the behaviour we expect, as we have $f_{\mathrm{smooth}}^1 \in \B^{\alpha,p}_p(\R^2)$ for any $\alpha < 1 + \nicefrac{1}{p}$.
Moreover, we see that the error is smaller for $\nu = 5$ than for $\nu = 7$.
This again fits to our expectations as the constant $c_{\alpha,K}$ is smaller for $\nu = 5$ for $\alpha \in \{2,\nicefrac{7}{4},\nicefrac{3}{2},\nicefrac{5}{4}\}$, see Table~\ref{tab:c_alpha_K}.

Recall that the smooth filter of order $\nu \geq 3$ satisfies the moment conditions
\begin{equation*}
\int_{\R^2} x^{j_1} y^{j_2} \, K(x,y) \: \d (x,y) = 0
\quad \forall \, 1 \leq |\boldsymbol{j}| \leq n
\end{equation*}
only for $n = 1$ so that Theorem~\ref{theo:Besov_large} can only predict a decrease of the error with rate $L^{-2}$ at most.
Thus, our theory predicts saturation of the error decay rate if the smoothness $\alpha$ of the target function is larger than $2$.
To illustrate this saturation numerically, we use the smooth phantom of order $\sigma = 2$, which satisfies $f_{\mathrm{smooth}}^2 \in \B^{\alpha,p}_p(\R^2)$ for any $\alpha < 2 + \nicefrac{1}{p}$.
Figure~\ref{fig:approximation_error_smooth2} shows the discrete $\L^p$-norm of the corresponding FBP approximation error exemplarily for $p \in \{1,4\}$, where we use the smooth filter with $\nu = 5$.
We indeed observe that the error decreases as $L^{-2}$.
This is true for all $1 \leq p < \infty$ and $\nu \geq 3$.
Hence, the predicted saturation of the decay rate is observable in numerical experiments.

\medskip

In summary, our reported numerical results totally comply with our theoretical findings concerning the FBP approximation error.

\medskip

In our second set of numerical experiments we investigate the FBP data error
\begin{equation*}
f_L - f_L^\delta
\end{equation*}
on the rectangular imaging domain $\Omega = [-1,1]^2$, where
\begin{equation*}
f_L^\delta = \frac{1}{2} \, \Back\bigl(\Fourier^{-1} A_L * g^\delta\bigr)
\end{equation*}
denotes the approximate FBP reconstruction from noisy Radon measurements $g^\delta$ with noise level $\delta > 0$.
To this end, we assume that we are given noisy measurements
\begin{equation*}
\big\{g^\delta_{m,n} \mid -M \leq m \leq M, ~ 0 \leq n \leq N-1\big\}
\end{equation*}
that satisfy
\begin{equation*}
\|\Radon f - g^\delta\|_{\ell^p} \leq \delta.
\end{equation*}

More precise, in our numerical simulations we use additive white Gaussian noise with noise-level
\begin{equation*}
\delta = 0.1 \cdot m_{\Radon f},
\end{equation*}
where
\begin{equation*}
m_{\Radon f} = \frac{1}{(2M+1)N} \sum_{m=-M}^M \sum_{n=0}^{N-1} |(\Radon f)_{m,n}|
\end{equation*}
is the arithmetic mean of the absolute values of the Radon samples of $f$.
Moreover, we again use the smooth filter with parameter $\nu \in \{5,7\}$.
Recall that, according to Theorem~\ref{theo:data_error}, the $\L^p$-norm of the data error can be estimated as
\begin{equation*}
\|f_L - f_L^\delta\|_{\L^p(\Omega)} \leq \frac{1}{2\pi^{\nicefrac{1}{p}}} \, \diam(\Omega)^{\nicefrac{1}{p}} \, \|\Fourier^{-1} A\|_{\L^1(\R)} \, L \, \delta,
\end{equation*}
where we have $\|\Fourier^{-1} A\|_{\L^1(\R^2)} = 0.2976$ for $\nu = 5$ and $\|\Fourier^{-1} A\|_{\L^1(\R^2)} = 0.2541$ for $\nu = 7$.
Consequently, we expect the error to be smaller for $\nu = 7$ and that the error increases with increasing $L$ with rate $L^1$, which is independent of the integrability parameter $p$.

\begin{figure}[p]
\centering
\includegraphics{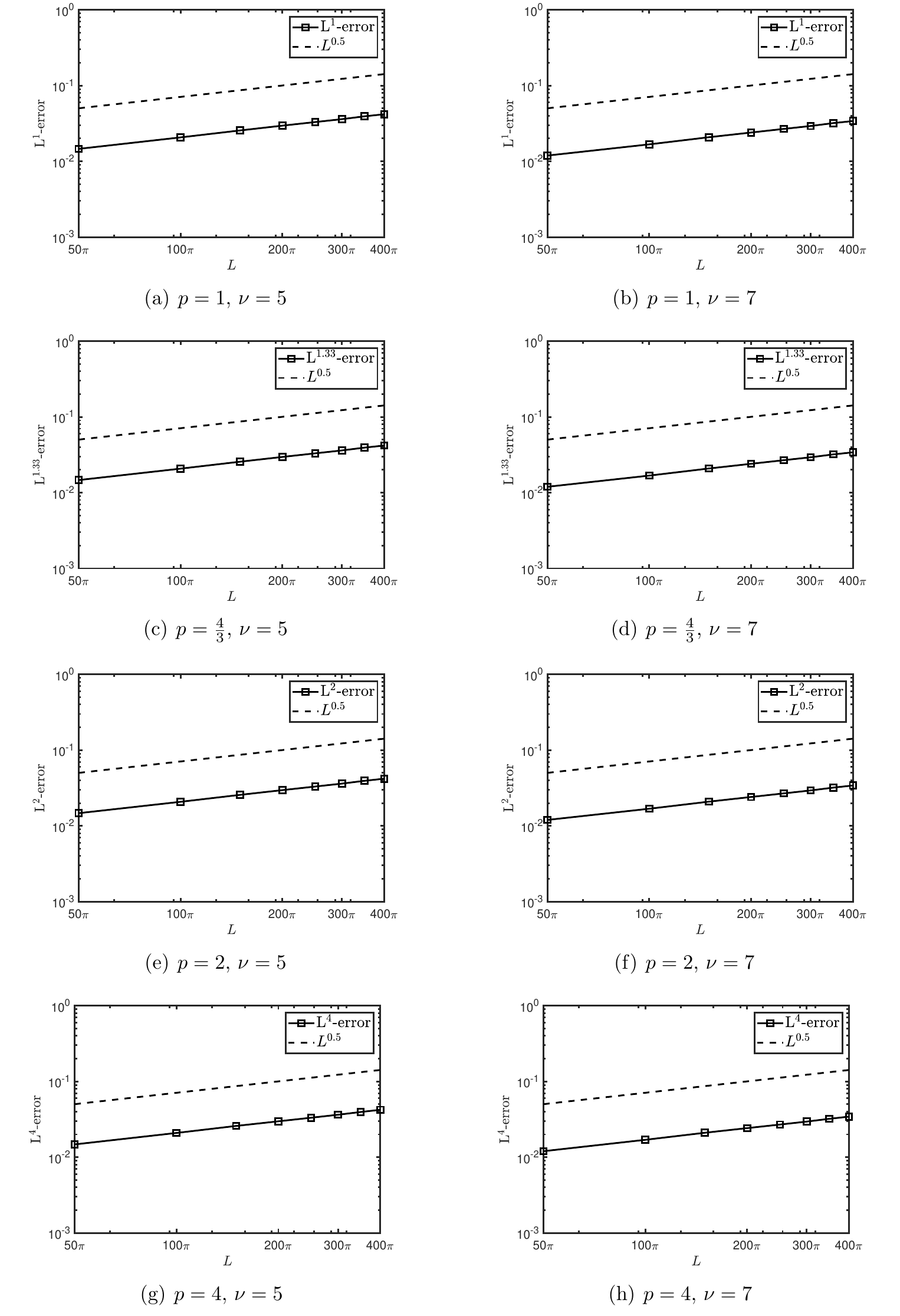}
\caption{$\L^p$-data error for Shepp-Logan phantom and smooth filter.}
\label{fig:data_error_shepp-logan}
\end{figure}

\begin{figure}[p]
\centering
\includegraphics{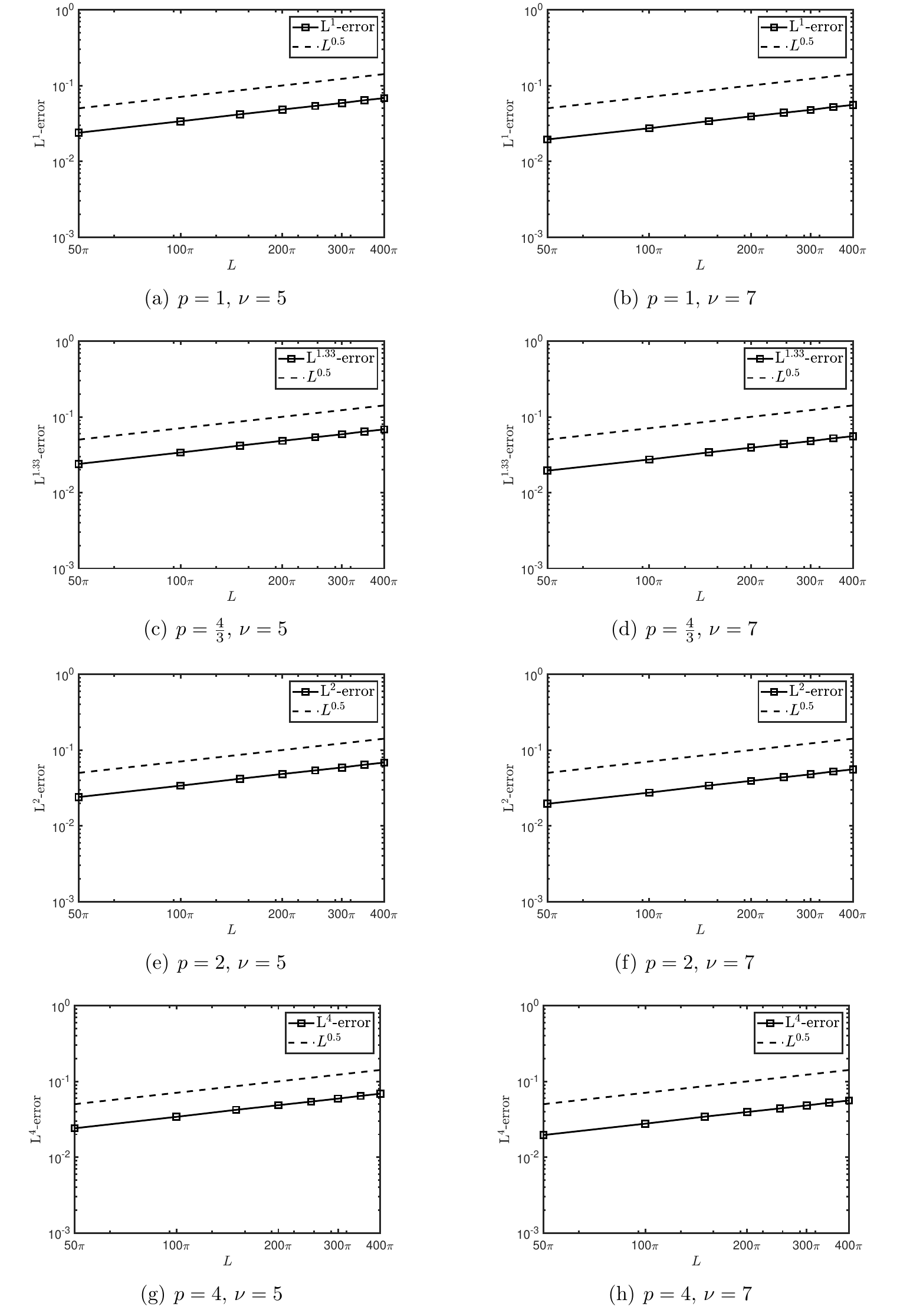}
\caption{$\L^p$-data error for smooth phantom ($\sigma = 1$) and smooth filter.}
\label{fig:data_error_smooth1}
\end{figure}

Figure~\ref{fig:data_error_shepp-logan} shows the discrete $\L^p$-norm of the FBP data error for the Shepp-Logan phantom for $p \in \{1,\nicefrac{4}{3},2,4\}$ as a function of the bandwidth $L$ in logarithmic scales.
The results for the smooth phantom of order $\sigma = 1$ are summarized in Figure~\ref{fig:data_error_smooth1}.
In all cases, we observe that the data error increases with $L$ with rate $L^{\nicefrac{1}{2}}$, which is indeed independent of $p$.
However, the growth rate is overestimated in Theorem~\ref{theo:data_error}, where the error bound increases with rate $L^1$.
Moreover, our numerical experiments  show that the data error is indeed smaller for $\nu = 7$ than for $\nu = 5$, as suggested by our error estimate.

We wish to remark that the reported error behaviour can also be observed in numerical experiments with other choices of $p$, $\nu$ and $\delta$.
Moreover, we note that the correct growth rate $L^{\nicefrac{1}{2}}$ of the FBP data error was derived in~\cite{Beckmann2019a} for the case $p = 2$.

\section{Conclusion}

In this paper, we have analysed the approximation and the total reconstruction error of the FBP method for CT reconstructions from parallel beam data. 
Our results depend on the smoothness and the flatness of the filter's window function near $0$.
Moreover, the rate of convergence depends on the smoothness of the true solution $f \in \B^{\alpha,p}_q(\R^2)$.

The convergence results cover the cases $1 \le p,q \le \infty$ and $0 < \alpha < 1$, resp. $1 < \alpha < \infty, ~ \alpha \notin \N$. Integer values for $\alpha$ require slightly different definitions of the Besov spaces.
This case is not covered by our analysis.

The numerical examples with phantoms of different Besov smoothness confirm the theoretical convergence rates of the approximation error, which require to link the flatness of the window function at $0$ with the smoothness of the target function for optimal convergence rates.
The data error, however, is overestimated by our theory and requires further in-depth research.

\section*{Acknowledgement}

J. Nickel acknowledges the support by the Federal Ministry of Education and Research (BMBF project no. 05M20LBB, DELETO). P. Maass acknowledges the support by the Deutsche Forschungsgemeinschaft (DFG) within the Research Training Group GRK 2224/2 $\pi^3$: ''Parameter Identification -- Analysis, Algorithms, Applications''.
 
J. Nickel and P. Maass acknowledge that major parts of the analytical sections as well as parts of the numerical simulations were done by M. Beckmann.

\section*{Appendix A}

We include the proof of Lemma~\ref{lem:embedding_p}.

\begin{proof}
Let $c > 1$.
Then, we can partition $(0,\infty)$ into the disjoint intervals $(c^{-k-1},c^{-k}]$, $k \in \Z$, and obtain
\begin{equation*}
\left(\int_0^\infty \big(t^{-\alpha} \, g(t)\big)^p \: \frac{\d t}{t} \right)^{\nicefrac{1}{p}} = \left(\sum_{k\in\Z} \int_{c^{-k-1}}^{c^{-k}} \big(t^{-\alpha} \, g(t)\big)^p \: \frac{\d t}{t} \right)^{\nicefrac{1}{p}}.
\end{equation*}
Since $g$ is monotonically increasing, we have
\begin{equation*}
\frac{g(c^{-k-1})}{c^{-k\alpha}} \leq \frac{g(t)}{t^\alpha} \leq \frac{g(c^{-k})}{c^{(-k-1)\alpha}}
\quad \forall \, t \in (c^{-k-1},c^{-k}]
\end{equation*}
such that
\begin{eqnarray*}
\fl \int_0^\infty \Big(\frac{g(t)}{t^\alpha}\Big)^p \: \frac{\d t}{t} & \leq \sum_{k\in\Z} \int_{c^{-k-1}}^{c^{-k}} \Big(\frac{g(c^{-k})}{c^{(-k-1)\alpha}}\Big)^p \: \frac{\d t}{t} = \sum_{k\in\Z} \Big(\frac{g(c^{-k})}{c^{-k\alpha}}\Big)^p \int_{c^{-k-1}}^{c^{-k}} c^{p\alpha} \: \frac{\d t}{t} \\
& = \sum_{k\in\Z} c^{p\alpha} \, \log(c) \,\Big(\frac{g(c^{-k})}{c^{-k\alpha}}\Big)^p.
\end{eqnarray*}
On the other hand, we have
\begin{eqnarray*}
\fl \int_0^\infty \Big(\frac{g(t)}{t^\alpha}\Big)^q \: \frac{\d t}{t} & = \sum_{k\in\Z} \int_{c^{-k}}^{c^{-k+1}} \Big(\frac{g(t)}{t^\alpha}\Big)^q \: \frac{\d t}{t} \geq \sum_{k\in\Z} \int_{c^{-k}}^{c^{-k+1}} \Big(\frac{g(c^{-k})}{c^{(-k+1)\alpha}}\Big)^q \: \frac{\d t}{t} \\
& = \sum_{k\in\Z} \Big(\frac{g(c^{-k})}{c^{-k\alpha}}\Big)^q \int_{c^{-k}}^{c^{-k+1}} c^{-q\alpha} \: \frac{\d t}{t} = \sum_{k\in\Z} c^{-q\alpha} \, \log(c) \, \Big(\frac{g(c^{-k})}{c^{-k\alpha}}\Big)^q.
\end{eqnarray*}
Thus, using the fact that $\ell^p \subset \ell^q$ for all $1 \leq p < q < \infty$ with
$\|a\|_{\ell^q} \leq \|a\|_{\ell^p}$ for all $a \in \ell^q$, we finally obtain
\begin{eqnarray*}
\left(\int_0^\infty \Big(\frac{g(t)}{t^\alpha}\Big)^p \: \frac{\d t}{t} \right)^{\nicefrac{1}{p}}
%& \leq \left(\sum_{k\in\Z} \bigg(c^\alpha \, \log(c)^{\nicefrac{1}{p}} \, \frac{g(c^{-k})}{c^{-k\alpha}}\bigg)^p\right)^{\nicefrac{1}{p}} \\
& \leq \left(\sum_{k\in\Z} \bigg(c^\alpha \, \log(c)^{\nicefrac{1}{p}} \, \frac{g(c^{-k})}{c^{-k\alpha}}\bigg)^q\right)^{\nicefrac{1}{q}} \\
& = c^{2\alpha} \, \log(c)^{\nicefrac{1}{p}-\nicefrac{1}{q}} \left(\sum_{k\in\Z} c^{-q\alpha} \, \log(c) \, \Big(\frac{g(c^{-k})}{c^{-k\alpha}}\Big)^q\right)^{\nicefrac{1}{q}} \\
& \leq c^{2\alpha} \, \log(c)^{\nicefrac{1}{p}-\nicefrac{1}{q}} \left(\int_0^\infty \Big(\frac{g(t)}{t^\alpha}\Big)^q \: \frac{\d t}{t} \right)^{\nicefrac{1}{q}},
\end{eqnarray*}
as stated.
\end{proof}

\section*{Appendix B}

We include the proof of Lemma~\ref{lem:limit}.

\begin{proof}
Let $c > 1$.
Assume that
\begin{equation*}
0 < \left(\int_0^\infty \big(t^{-\alpha} \, g(t)\big)^q \: \frac{\d t}{t} \right)^{\nicefrac{1}{q}} < \infty.
\end{equation*}
Otherwise, the stated estimate is trivially true.
Then, by Lemma~\ref{lem:embedding_p} for all $1 \leq q < p < \infty$ holds that
\begin{eqnarray*}
\left(\int_0^\infty \Big(\frac{g(t)}{t^\alpha}\Big)^p \: \frac{\d t}{t} \right)^{\nicefrac{1}{p}} & \leq c^{2\alpha} \, \log(c)^{\nicefrac{1}{p}-\nicefrac{1}{q}} \left(\int_0^\infty \Big(\frac{g(t)}{t^\alpha}\Big)^q \: \frac{\d t}{t} \right)^{\nicefrac{1}{q}} \\
& \to c^{2\alpha} \, \log(c)^{-\nicefrac{1}{q}} \left(\int_0^\infty \Big(\frac{g(t)}{t^\alpha}\Big)^q \: \frac{\d t}{t} \right)^{\nicefrac{1}{q}} < \infty
\end{eqnarray*}
for $p \to \infty$ and it suffices to prove that
\begin{equation*}
\left(\int_0^\infty \Big(\frac{g(t)}{t^\alpha}\Big)^p \: \frac{\d t}{t} \right)^{\nicefrac{1}{p}} ~\to~ \sup_{t > 0} \frac{g(t)}{t^\alpha}
\quad \mbox{ for } \quad
p \to \infty.
\end{equation*}
Since $g$ is monotonically increasing and
\begin{equation*}
\left(\int_0^\infty \Big(\frac{g(t)}{t^\alpha}\Big)^q \: \frac{\d t}{t} \right)^{\nicefrac{1}{q}} < \infty,
\end{equation*}
we have that
\begin{equation*}
M = \sup_{t>0} \frac{g(t)}{t^\alpha} < \infty.
\end{equation*}
Now fix $0 < \delta < M$ and consider
\begin{equation*}
D = \left\{t > 0 \mid \frac{g(t)}{t^\alpha} \geq M-\delta\right\}.
\end{equation*}
By definition, we have $\lambda(D) > 0$ and obtain
\begin{equation*}
\fl \infty > \left(\int_0^\infty \Big(\frac{g(t)}{t^\alpha}\Big)^q \: \frac{\d t}{t} \right)^{\nicefrac{1}{q}} \geq \left(\int_D \Big(\frac{g(t)}{t^\alpha}\Big)^q \, t^{-1} \: \d t \right)^{\nicefrac{1}{q}} \geq (M-\delta) \left(\int_D t^{-1} \: \d t \right)^{\nicefrac{1}{q}}.
\end{equation*}
In particular, we have
\begin{equation*}
0 < \int_D t^{-1} \: \d t < \infty
\end{equation*}
so that
\begin{equation*}
\left(\int_D t^{-1} \: \d t \right)^{\nicefrac{1}{p}} \to 1
\quad \mbox{ for } \quad
p \to \infty.
\end{equation*}
This shows that
\begin{equation*}
\liminf_{p \to \infty} \left(\int_0^\infty \Big(\frac{g(t)}{t^\alpha}\Big)^p \: \frac{\d t}{t} \right)^{\nicefrac{1}{p}} \geq \left(\sup_{t>0} \frac{g(t)}{t^\alpha}\right)-\delta
\end{equation*}
and, since $0 < \delta < M$ was arbitrary,
\begin{equation*}
\liminf_{p \to \infty} \left(\int_0^\infty \Big(\frac{g(t)}{t^\alpha}\Big)^p \: \frac{\d t}{t} \right)^{\nicefrac{1}{p}} \geq \sup_{t>0} \frac{g(t)}{t^\alpha}.
\end{equation*}
On the other hand, for $p > q$ we have
\begin{equation*}
\fl \left(\int_0^\infty \Big(\frac{g(t)}{t^\alpha}\Big)^p \: \frac{\d t}{t} \right)^{\nicefrac{1}{p}} \leq \left(\sup_{t>0} \frac{g(t)}{t^\alpha}\right)^{1-\nicefrac{q}{p}} \left(\int_0^\infty \Big(\frac{g(t)}{t^\alpha}\Big)^q \: \frac{\d t}{t} \right)^{\nicefrac{1}{p}} ~\to~ \sup_{t>0} \frac{g(t)}{t^\alpha}
\end{equation*}
for $p \to \infty$, since
\begin{equation*}
0 < \int_0^\infty \Big(\frac{g(t)}{t^\alpha}\Big)^q \: \frac{\d t}{t} < \infty.
\end{equation*}
Consequently, we also have
\begin{equation*}
\limsup_{p \to \infty} \left(\int_0^\infty \Big(\frac{g(t)}{t^\alpha}\Big)^p \: \frac{\d t}{t} \right)^{\nicefrac{1}{p}} \leq \sup_{t>0} \frac{g(t)}{t^\alpha}
\end{equation*}
so that in total
\begin{equation*}
\lim_{p \to \infty} \left(\int_0^\infty \Big(\frac{g(t)}{t^\alpha}\Big)^p \: \frac{\d t}{t} \right)^{\nicefrac{1}{p}} = \sup_{t>0} \frac{g(t)}{t^\alpha}.
\end{equation*}
Hence, with Lemma~\ref{lem:embedding_p} we can conclude that
\begin{equation*}
\sup_{t > 0} t^{-\alpha} \, g(t) \leq c^{2\alpha} \, \log(c)^{-\nicefrac{1}{q}} \left(\int_0^\infty \big(t^{-\alpha} \, g(t)\big)^q \: \frac{\d t}{t} \right)^{\nicefrac{1}{q}}
\end{equation*}
and the proof is complete.
\end{proof}

\section*{Appendix C}

We include the proof of Lemma~\ref{lem:limit}.

\begin{proof}
Since $g$ is increasing and bounded from above, $g$ is Lebesgue measurable and convergent for $t \to \infty$, i.e., there exists $l \in [0,\infty)$ such that $l = \lim_{t \to \infty} g(t)$.
If $l = 0$, we have $g \equiv 0$ and the statement is trivially true.
Thus, assume that $l \in (0,\infty)$.
To prove the statement, we show \\
\begin{minipage}{0.5\textwidth}
\begin{itemize}
\item[(1)] $\displaystyle \liminf_{\alpha \searrow 0} \left(\alpha q \int_0^\infty \Big(\frac{g(t)}{t^\alpha}\Big)^q \: \frac{\d t}{t} \right)^{\nicefrac{1}{q}} \geq l$,
\end{itemize}
\end{minipage}%
\begin{minipage}{0.5\textwidth}
\begin{itemize}
\item[(2)] $\displaystyle \limsup_{\alpha \searrow 0} \left(\alpha q \int_0^\infty \Big(\frac{g(t)}{t^\alpha}\Big)^q \: \frac{\d t}{t} \right)^{\nicefrac{1}{q}} \leq l$.
\end{itemize}
\end{minipage}

\medskip

\noindent
\underline{ad (1):} Fix $0 < l_0 < l$.
Then, there is $T_0 > 0$ such that
\begin{equation*}
g(t) \geq l_0
\quad \forall \, t \geq T_0.
\end{equation*}
With this, for $0 < \alpha \leq \sigma$, we obtain
\begin{equation*}
\alpha q \int_0^\infty \Big(\frac{g(t)}{t^\alpha}\Big)^q \: \frac{\d t}{t} = \alpha q \int_0^{T_0} \Big(\frac{g(t)}{t^\alpha}\Big)^q \: \frac{\d t}{t} + \alpha q \int_{T_0}^\infty \Big(\frac{g(t)}{t^\alpha}\Big)^q \: \frac{\d t}{t} = I_1 + I_2
\end{equation*}
with
\begin{equation*}
\fl 0 \leq I_1 = \alpha q \int_0^{T_0} t^{(\sigma - \alpha) q} \, \Big(\frac{g(t)}{t^\sigma}\Big)^q \: \frac{\d t}{t} \leq \alpha q \, T_0^{(\sigma - \alpha) q} \int_0^\infty \Big(\frac{g(t)}{t^\sigma}\Big)^q \: \frac{\d t}{t} \to 0
\quad \mbox{ for } \quad
\alpha \searrow 0
\end{equation*}
and
\begin{equation*}
I_2 \geq \alpha q \, l_0^q \int_{T_0}^\infty \Big(\frac{1}{t^\alpha}\Big)^q \: \frac{\d t}{t} = l_0^q \, T_0^{-\alpha q} \to l_0^q
\quad \mbox{ for } \quad
\alpha \searrow 0.
\end{equation*}
In particular, taking the limit $l_0 \nearrow l$ gives
\begin{equation*}
\lim_{\alpha \searrow 0} I_1 = 0
\quad \mbox{ and } \quad
\liminf_{\alpha \searrow 0} I_2 \geq l^q.
\end{equation*}

\noindent
\underline{ad (2):} Fix $l_1 > l$.
Then, there is $T_1 > 0$ such that
\begin{equation*}
g(t) \leq l_1
\quad \forall \, t \geq T_1.
\end{equation*}
With this, for $0 < \alpha \leq \sigma$, we obtain
\begin{equation*}
\alpha q \int_0^\infty \Big(\frac{g(t)}{t^\alpha}\Big)^q \: \frac{\d t}{t} = \alpha q \int_0^{T_1} \Big(\frac{g(t)}{t^\alpha}\Big)^q \: \frac{\d t}{t} + \alpha q \int_{T_1}^\infty \Big(\frac{g(t)}{t^\alpha}\Big)^q \: \frac{\d t}{t} = I_1 + I_2
\end{equation*}
with
\begin{equation*}
\fl I_1 = \alpha q \int_0^{T_1} t^{(\sigma - \alpha) q} \, \Big(\frac{g(t)}{t^\sigma}\Big)^q \: \frac{\d t}{t} \leq \alpha q \, T_1^{(\sigma - \alpha) q} \int_0^\infty \Big(\frac{g(t)}{t^\sigma}\Big)^q \: \frac{\d t}{t} \to 0
\quad \mbox{ for } \quad
\alpha \searrow 0
\end{equation*}
and
\begin{equation*}
I_2 \leq \alpha q \, l_1^q \int_{T_1}^\infty \Big(\frac{1}{t^\alpha}\Big)^q \: \frac{\d t}{t} = l_1^q \, T_1^{-\alpha q} \to l_1^q
\quad \mbox{ for } \quad
\alpha \searrow 0.
\end{equation*}
In particular, taking the limit $l_1 \searrow l$ gives
\begin{equation*}
\lim_{\alpha \searrow 0} I_1 = 0
\quad \mbox{ and } \quad
\limsup_{\alpha \searrow 0} I_2 \leq l^q.
\end{equation*}

With (1) and (2) we now have
\begin{equation*}
l \leq \liminf_{\alpha \searrow 0} \left(\alpha q \int_0^\infty \Big(\frac{g(t)}{t^\alpha}\Big)^q \: \frac{\d t}{t} \right)^{\nicefrac{1}{q}} \leq \limsup_{\alpha \searrow 0} \left(\alpha q \int_0^\infty \Big(\frac{g(t)}{t^\alpha}\Big)^q \: \frac{\d t}{t} \right)^{\nicefrac{1}{q}} \leq l
\end{equation*}
so that
\begin{equation*}
\lim_{\alpha \searrow 0} \left(\alpha q \int_0^\infty \Big(\frac{g(t)}{t^\alpha}\Big)^q \: \frac{\d t}{t} \right)^{\nicefrac{1}{q}} = l = \lim_{t \to \infty} g(t),
\end{equation*}
as stated.
\end{proof}

\end{document}